\documentclass[12pt,sort&compress]{elsarticle}

\usepackage{amsthm,amsmath,txfonts}
\usepackage[colorlinks,citecolor=blue]{hyperref}
\usepackage{graphicx}
\usepackage{subfigure}
\usepackage{tikz}
\usepackage{anysize}
\usepackage{setspace}
\usetikzlibrary{arrows,shapes,chains,3d}
\marginsize{3.5cm}{3cm}{3cm}{3cm}
\newtheorem{theorem}{Theorem}[section]
\newtheorem{lemma}[theorem]{Lemma}
\newtheorem{definition}[theorem]{Definition}
\newtheorem{proposition}[theorem]{Proposition}
\newtheorem{corollary}[theorem]{Corollary}
\theoremstyle{definition}
\newtheorem{remark}{Remark}[section]
\newtheorem*{claim}{Claim}
\newtheorem{example}{Example}[section]
\numberwithin{equation}{section}
\numberwithin{figure}{section}
\allowdisplaybreaks
\hypersetup{CJKbookmarks={true}}
\begin{document}

\begin{frontmatter}
\title{Fractional Fourier transforms on $L^p$  and applications\tnoteref{tn}}
\tnotetext[tn]{The second and fourth authors were partially supported by the National Natural Science Foundation of China (Nos. 11671185, 11701251 and 11771195) and the Natural Science Foundation of Shandong Province (Nos. ZR2017BA015, ZR2018LA002 and ZR2019YQ04). The third author was partially supported by the Simons Foundation (No. 315380).}

\author[swu]{Wei Chen}
\ead{cwei1990@126.com}
\author[lyu]{Zunwei Fu}
\ead{fuzunwei@eoyu.com}
\author[mu]{Loukas Grafakos}
\ead{grafakosl@missouri.edu}
\author[lyu]{Yue Wu}
\ead{wuyue@lyu.edu.cn}
\address[swu]{College of Information Technology, The University of Suwon, Hwaseong-si 18323, South Korea}
\address[lyu]{School of Mathematics and Statistics, Linyi University, Linyi 276000, China}
\address[mu]{Department of Mathematics, University of Missouri, Columbia MO 65211, USA}

\begin{abstract}
This paper is devoted to the $L^p(\mathbb R)$ theory of the fractional Fourier transform (FRFT) for $1\le p < 2$. In view of the special structure of  the  FRFT,  we study FRFT properties of    $L^1$ functions, via the introduction of a suitable chirp operator. However, in the  $L^1(\mathbb{R})$ setting, problems of convergence arise  even when basic    manipulations of functions are performed. We overcome such issues and study the FRFT inversion problem via approximation by suitable means, such as the fractional Gauss and Abel means. We also obtain the regularity of   fractional convolution and   results on pointwise convergence of FRFT means.  Finally we discuss  $L^p$ multiplier results and a Littlewood-Paley theorem associated with FRFT.
\end{abstract}
\end{frontmatter}

\begin{footnotesize}
\begin{spacing}{-1.0}
\vspace{-20pt}
  \tableofcontents
  \end{spacing}
\end{footnotesize}

\section{Introduction}

\label{sect:intro}

In classical Fourier analysis three  important classes of operators arise:
maximal averages, singular integrals, and oscillatory
integrals. The Hardy-Littlewood maximal operator, the Hilbert transform and
the Fourier transform, respectively, are prime examples of these classes
of operators. In recent decades
fractional versions of the first two types of operators  have been
widely studied, but less attention has been paid to the
mathematical theory of the fractional Fourier transform. In this paper we undertake this task,
which is strongly motivated by the important role it plays in practical applications.

The Fourier transform is one of the most important and powerful tools in
theoretical and applied mathematics. Mainly driven
by the need to analyze and
process non-stationary signals, the Fourier transform of fractional
order has been proposed and developed by several scholars.
At present, the fractional Fourier transform
(FRFT for short) has found applications in many aspects of scientific research and
engineering technology, such as swept filter, artificial neural network,
wavelet transform, time-frequency analysis, time-varying filtering, complex
transmission and so on (see, e.g., \cite{S2011,T2010,D2001,N2003,M2002,Y2003}%
). In addition, it was also used widely in fields of solving partial
differential equations (cf., \cite{namias,Lohmann93}), quantum mechanics (cf.,
\cite{namias,PhysRevLett72}), diffraction theory and optical transmission
(cf., \cite{Ozak}), optical system and optical signal processing (cf.,
\cite{Bernardo94,OZAKTAS1995119,Liu97}), optical image processing (cf.,
\cite{Lohmann93,Liu97}), etc.

The FRFT is a fairly old mathematical topic. It dates back to work by Wiener
\cite{W1929} in 1929, but it was not until the
past three decades that significant attention was paid to this object
starting with Namias' work \cite{namias}
in 1980. The approach used by Namias relies primarily on eigenfunction
expansions. For suitable functions $f$ on the line, the classical Fourier transform
$\mathcal{F}$ is defined as follows%
\begin{equation}
(\mathcal{F}f)(x)=\int_{-\infty}^{+\infty}f(t)e^{-2\pi ixt}\mathrm{d}t.
\label{ft}%
\end{equation}
It is known that $\mathcal{F}$ is a homeomorphism on $L^{2}(\mathbb{R})$ and
has eigenvalues
\[
\lambda_{n}=e^{in\pi/2},\quad n=0,1,2,\ldots.
\]
with corresponding eigenfunctions
\[
\psi_{n}(x)=e^{-x^{2}/2}H_{n}(x),
\]
where $H_{n}$ is the Hermite polynomial of degree $n$ (see \cite{GL}). Since
$\{\psi_{n}\}$ is an orthonrmal basis of $L^{2}(\mathbb{R})$, it follows that
\[
\mathcal{F}f=\sum_{n}e^{-in\pi/2}(f,\psi_{n})_{2}\psi_{n},\quad\forall f\in
L^{2}(\mathbb{R}).
\]
This naturally leads to the definition of the fractional order operators $\{\mathcal{F}_{\alpha}\}$
for $\alpha\in\mathbb{R}$ via%
\[
\mathcal{F}_{\alpha}f=\sum_{n}e^{-in\alpha}(f,\psi_{n})_{2}\psi_{n}%
,\quad\forall f\in L^{2}(\mathbb{R}).
\]
It is clear that $\mathcal{F}_{\alpha}=\mathcal{F}$ when $\alpha=\pi/2$.

In 1987, McBride and Kerr   \cite{MK1987} provided a   rigorous definition
on the Schwartz space $S(\mathbb{R})$
of the FRFT in   integral form based on a modification of Namias' fractional operators.
For $\left\vert
\alpha\right\vert \in\left(  0,\pi\right)  $, McBride and Kerr   \cite{MK1987} defined the FRFT by%
\begin{equation}
(\mathcal{F}_{\alpha}f)(x)=\frac{e^{i\left(  \hat{\alpha}\pi/4-\alpha
/2\right)  }}{\sqrt{\left\vert \sin\alpha\right\vert }}e^{i\pi x^{2}\cot
\alpha}\int_{-\infty}^{+\infty}f(t)e^{-\pi i(2xt\csc\alpha-t^{2}\cot\alpha
)}\mathrm{d}t, \label{intro:fa}%
\end{equation}
where $\hat{\alpha}=\mathrm{sgn}\left(  \sin\alpha\right)  $. The definition
  extends to all $\alpha\in\mathbb{R}$ by periodicity. Moreover, these authors proved that
\begin{theorem}
[\cite{MK1987}]\label{prop:S}For  all $ f\in S(\mathbb{R})$ and all
$ \alpha,\beta\in\mathbb{R}$ we have
\begin{enumerate}
[(i)]
\item $\mathcal{F}_{\alpha}:S(\mathbb{R})\rightarrow S(\mathbb{R})$ is a homeomorphism;

\item $\mathcal{F}_{\alpha}\mathcal{F}_{\beta}f=\mathcal{F}_{\alpha+\beta}f$.
\end{enumerate}
\end{theorem}

Later, Kerr \cite{Kerr1988} studied the $L^{2}(\mathbb{R})$ theory of $\mathcal{F}_{\alpha}$. He gave the definition of FRFT on $L^{2}%
(\mathbb{R})$ by interpreting (\ref{intro:fa}) as follows:%
\begin{equation}
(\mathcal{F}_{\alpha}f)(x)=\frac{e^{i\left(  \hat{\alpha}\pi/4-\alpha
/2\right)  }}{\sqrt{\left\vert \sin\alpha\right\vert }}e^{i\pi x^{2}\cot
\alpha}\lim_{R\rightarrow\infty}\int_{-R}^{R}f(t)e^{-\pi i(2xt\csc\alpha
-t^{2}\cot\alpha)}\mathrm{d}t \label{intro:fa2}%
\end{equation}
and proved the following result.

\begin{theorem}
[\cite{Kerr1988}]\label{prop:L2}For  all $f,g\in L^{2}(\mathbb{R})$ and all
$ \alpha,\beta\in\mathbb{R}$ we have
\begin{enumerate}
[(i)]
\item $\mathcal{F}_{\alpha}:L^{2}(\mathbb{R})\rightarrow L^{2}(\mathbb{R})$ is
a homeomorphism;

\item $\left\Vert \mathcal{F}_{\alpha}f\right\Vert _{2}=\left\Vert
f\right\Vert _{2}$;

\item $\mathcal{F}_{\alpha}\mathcal{F}_{\beta}f=\mathcal{F}_{\alpha+\beta}f$;

\item $\int_{-\infty}^{+\infty}(\mathcal{F}_{\alpha}f)(x)g(x)\mathrm{d}%
x=\int_{-\infty}^{+\infty}f(x)(\mathcal{F}_{\alpha}g)(x)\mathrm{d}x$;

\item $\{\mathcal{F}_{\alpha}:\alpha\in\mathbb{R}\}$ is a strongly continuous
unitary group of operators on $L^{2}(\mathbb{R)}$.
\end{enumerate}
\end{theorem}

For additional work in this area, refer to \cite{KF1988,KLM}, etc..

In an attempt to take the theory of FRFT beyond
$S(\mathbb{R})$ or $L^{2}(\mathbb{R)}$, we discuss in this paper  (Section \ref{sect:Lp})
the behavior of FRFT on $L^{p}(\mathbb{R})$ for $1\leq p<2$. In Section \ref{sect:L1},
we discuss the elementary properties of FRFT on $L^{1}%
(\mathbb{R})$. Section \ref{sect:invs} is devoted to the problem of FRFT inversion,   which
is established via an approximation in terms of FRFT integral means. In Section \ref{sect:multiplier}, we discuss  $L^p$ multiplier results and a Littlewood-Paley theorem associated with FRFT. Using the language of time-frequency analysis, this   means that    an $L^{1}$ chirp signal,
whose FRFT is non-integrable, is recovered from the frequency domain
as a limit of the inverted Abel means of its FRFT; this is discussed in the last section.

\section{FRFT on $L^{1}(\mathbb{R})$}

\label{sect:L1}
It is natural to begin our exposition by
defining the FRFT on $L^{1}(\mathbb{R})$; our definition is like that in \cite{MK1987}. In
$L^{1}(\mathbb{R})$, problems of convergence arise when certain manipulations of
functions are performed and FRFT inversion is not possible.

\begin{definition}
For $f\in L^{1}(\mathbb{R})$ and $\alpha\in\mathbb{R}$, the \emph{fractional
Fourier transform} of order $\alpha$ of $f$ is defined by
\begin{equation}
(\mathcal{F}_{\alpha}f)(x)=\left\{
\begin{array}
[c]{ll}%
\int_{-\infty}^{+\infty}K_{\alpha}(x,t)f(t)\, \mathrm{d}t, & \alpha\neq n\pi,\quad
n\in\mathbb{N},\\
f(x), & \alpha=2n\pi,\\
f(-x), & \alpha=(2n+1)\pi,
\end{array}
\right.  \label{eq:fa}%
\end{equation}
where
\[
K_{\alpha}(x,t)=A_{\alpha}\exp\left[  2\pi i\left(  \frac{t^{2}}{2}\cot
\alpha-xt\csc\alpha+\frac{x^{2}}{2}\cot\alpha\right)  \right]
\]
is the kernel of FRFT and
\begin{equation}\label{defAa}
A_{\alpha}=\sqrt{1-i\cot\alpha}.
\end{equation}
\end{definition}

As the parameter $\alpha$ only appears as an argument of trigonometric functions  (see
(\ref{eq:fa})), it follows that $\mathcal{F}_{\alpha}$ is  $2\pi$-periodic
with respect to $\alpha$. Hence, throughout this paper we shall always assume
$\alpha\in\lbrack0,2\pi)$.

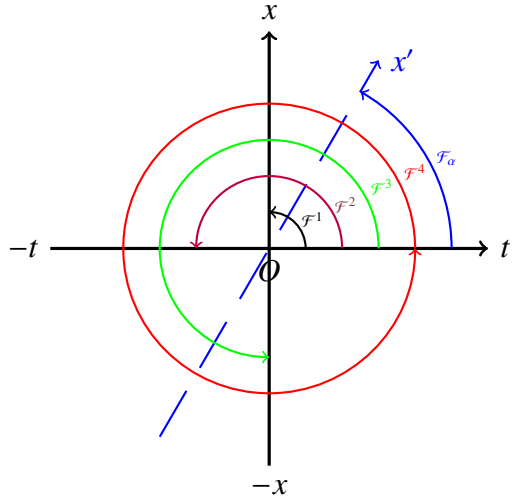
\begin{figure}[htb]
\centering
\begin{tikzpicture}[node distance=1cm,scale=0.48]
\draw[->,very thick](-6,0)node[left]{$-t$}--(6,0) node[right]{$t$};
\draw[->,very thick](0,-6)node[below]{$-x$}--(0,6)node[above]{$x$};
\draw[->,thick,blue,dash pattern=on 20pt off 10pt](-3,-5.2)--(3,5.2) node[right]{$x'$};
\draw[->,thick,blue] (5,0) arc (0:60:5);
\draw[->,thick,red] (4,0) arc (0:360:4);
\draw[->,thick, green] (3,0) arc (0:270:3);
\draw[->,thick,purple] (2,0) arc (0:180:2);
\draw[->,thick] (1,0) arc (0:90:1);
\node[thick,below] at (0,0){$O$};
\node[right] at (4.3,2.5){\tiny$\textcolor[rgb]{0,0,1}{\mathcal F_\alpha}$};
\node[right] at (3.4,2.1){\tiny$\textcolor[rgb]{1,0,0}{\mathcal F^4}$};
\node[right] at (2.5,1.7){\tiny$\textcolor[rgb]{0,1,0}{\mathcal F^3}$};
\node[right] at (1.5,1.2){\tiny$\textcolor[rgb]{0.5,0.25,0.25}{\mathcal F^2}$};
\node[right] at (0.6,0.8){\tiny$\mathcal F^1$};
\end{tikzpicture}
\caption{rotation of time-frequency domain}%
\label{fig:rot}%
\end{figure}

Notice now that when $n\in\mathbb{Z}$, $\mathcal{F}_{n\pi/2}f=\mathcal{F}%
^{n}f$, where $\mathcal{F}^{n}$ is the $n$th power of the classical Fourier
operator (\ref{ft}). Therefore, $\mathcal{F}_{\alpha}$ can be regarded as the
$s$th power of the Fourier transform, where $s=2\alpha/\pi$, that is,
\[
\mathcal{F}^{s}f=\mathcal{F}_{s\pi/2}f.
\]
Denote by $\mathcal{I}$ the identity operator and $\mathcal{P}$ the reflection
operator defined by $\mathcal{P}g(x) = g(-x)$. We can easily see that (Figure \ref{fig:rot})%
\begin{align*}
&  \mathcal{F}^{0}=\mathcal{F}_{0}=\mathcal{I};\\
&  \mathcal{F}^{1}=\mathcal{F}_{\pi/2}=\mathcal{F};\\
&  \mathcal{F}^{2}=\mathcal{F}_{\pi}=\mathcal{P};\\
&  \mathcal{F}^{3}=\mathcal{F}_{3\pi/2}=\mathcal{FP=PF};\\
&  \mathcal{F}^{4}=\mathcal{F}_{2\pi}=\mathcal{F}^{0}=\mathcal{I};\\
&  \mathcal{F}^{4n\pm s}=\mathcal{F}_{2n\pi\pm\alpha}=\mathcal{F}_{\pm\alpha
}=\mathcal{F}^{\pm s}.
\end{align*}

Every signal (or function) can be described indirectly and uniquely by a
Wigner distribution function (WDF). The classical Fourier transform
$\mathcal{F}$ lets the WDF rotate by an angle of $\pi/2$. Hence,
$\mathcal{F}u$ is the function corresponding to the WDF obtained by rotating the original
 WDF of $u$ by an angle   $\pi/2$. Analogously, the FRFT of order
$\alpha$ is the unique function whose WDF is obtained by rotating the  original
 WDF of $u$ by an angle   $\alpha$. We refer to \cite{Lohmann93} for more details on this.

\begin{example}
Define the following function on the line:
\[
f =\sum_{n=1}^\infty n \chi_{[n, n+\frac 1{n^3} )}.
\]
Using (\ref{eq:fa}), we can easily calculate the FRFT of this function:
\[
\left(  \mathcal{F}_{\alpha}f\right)  (x)=\frac{A_{\alpha}e^{i\pi x^{2}%
\cot\alpha}}{2\pi ix}\sum\limits_{n=1}^{\infty}ne^{-2n\pi ix}\left(
1-e^{-\frac{2\pi ix}{n^{3}}}\right)  ,
\]
where $A_\alpha$ is as in \eqref{defAa}.
This function lies in $L^{1}(\mathbb{R})$ but not in $L^{2}(\mathbb{R})$ as
\[
\int_{-\infty}^{+\infty}\left\vert f(t)\right\vert \mathrm{d}t=\sum
\limits_{n=1}^{\infty}\int_{n}^{n+\frac{1}{n^{3}}}n\,\mathrm{d}t=\sum
\limits_{n=1}^{\infty}\frac{1}{n^{2}}=\frac{\pi^{2}}{6}<\infty%
\]
and%
\[
\int_{-\infty}^{+\infty}\left\vert f(t)\right\vert ^{2}\mathrm{d}%
t=\sum\limits_{n=1}^{\infty}\int_{n}^{n+\frac{1}{n^{3}}}n^{2}\,\mathrm{d}%
t=\sum\limits_{n=1}^{\infty}\frac{1}{n} =\infty.
\]
\end{example}

\begin{remark}
\label{rm:M}Define the chirp operator $\mathcal{M}_{\alpha}$ acting on functions $
 \phi$  in $L^{1}(\mathbb{R})$ as follows:
\[
\mathcal{M}_{\alpha}\phi(x)=e^{\pi ix^{2}\cot\alpha}\phi(x).
\]
Then for $\alpha\neq n\pi$, let $A_\alpha$ be  as in \eqref{defAa}.
Then  the FRFT of $f\in L^{1}(\mathbb{R})$ can be
written as
\begin{align}
(\mathcal{F}_{\alpha}f)(x)  &  =A_{\alpha}e^{i\pi x^{2}\cot\alpha}%
\mathcal{F}[e^{i\pi t^{2}\cot\alpha}f(t)](x\csc\alpha)\nonumber\\
&  =A_{\alpha}\mathcal{M}_{\alpha}\mathcal{F}[\mathcal{M}_{\alpha}%
f(t)](x\csc\alpha). \label{eq:def}%
\end{align}

In view of (\ref{eq:def}), we see that the FRFT of a function (or signal)
$u(t)$ can be decomposed into four simpler operators,
according to the diagram of Figure \ref{fig:com}:

\begin{enumerate}
[(i)]

\item multiplication by a chirp signal, $g(t)=e^{\pi it^{2}\cot\alpha}u(t)$;

\item Fourier transform, $\hat{g}(x)=(\mathcal{F}g)(x)$;

\item scaling, $\tilde{g}(x)=\hat{g}(x\csc\alpha)$;

\item multiplication by a chirp signal, $(\mathcal{F}_{\alpha}u)(x)=A_{\alpha
}e^{\pi ix^{2}\cot\alpha}\tilde{g}(x)$.
\end{enumerate}
\end{remark}

\begin{figure}[htb]
\centering
{\small \begin{tikzpicture}[node distance=20mm,thick]
		\node (start){$u(t)$};
		\node[circle,draw,inner sep=2pt,right of=start] (MA){$\times$};
\node[below of =MA] (p1){{$e^{i\pi t^2 cot\alpha}$}};
		\node[rectangle,draw,right of=MA] (FT){$\mathcal F$};
		\node[rectangle,draw,right of=FT] (SC){Scaling};
		\node[circle,draw,inner sep=2pt,right of=SC] (MA2){$\times$};
\node[below of =MA2] (p2){$A_\alpha e^{i\pi x^2 cot\alpha}$};
		\node[right of=MA2] (end){$(\mathcal F _\alpha u)(x)$};
		\draw[->](start)--(MA);
		\draw[->](MA)--node[above]{$g(t)$}(FT);
		\draw[->](FT)--node[above]{$\hat{g}(x)$}(SC);
		\draw[->](SC)--node[above]{$\tilde{g}(x)$}(MA2);
		\draw[->](MA2)--(end);
		\draw[->](p1)--(MA);
		\draw[->] (p2) -- (MA2);
		\end{tikzpicture}
}\caption{the decomposition of the FRFT}%
\label{fig:com}%
\end{figure}
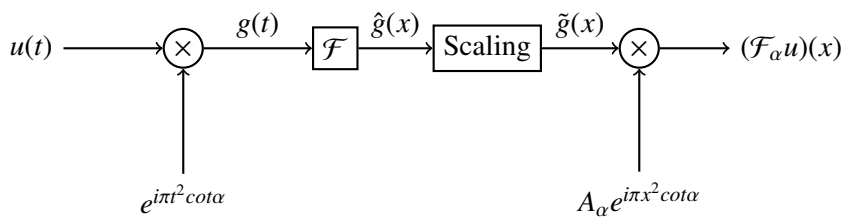

In view of the decomposition (\ref{eq:def}) of the FRFT, the boundedness
properties of the fractional Fourier operator $\mathcal{F}_{\alpha}$ is
largely the same of the classical Fourier operator $\mathcal{F}$. However, due
to the factors $e^{i\pi x^{2}\cot\alpha}$and $e^{i\pi t^{2}\cot\alpha}$, the
convergence properties are not trival.
We now discuss some basic properties of the FRFT on $L^{1}(\mathbb{R})$.

Firstly, we consider the behavior of FRFT at infinity. The following
  is the fractional version of the Riemann-Lebesgue lemma.

\begin{lemma}
[Riemann-Lebesgue lemma]\label{th:RL}For $f\in L^{1}(\mathbb{R})$, we have
that%
\[
\left\vert \left(  \mathcal{F}_{\alpha}f\right)  (x)\right\vert \rightarrow0
\]
as $x\rightarrow\infty$.
\end{lemma}

\begin{proof}
Since $\mathcal{M}_{\alpha}f\in L^{1}(\mathbb{R})$, then $|\mathcal{F}%
(\mathcal{M}_{\alpha}f)(x)|\rightarrow0$ as $x\rightarrow\infty$ by the
Riemann-Lebesgue lemma for the classical Fourier transform. Hence, it follows
from (\ref{eq:def}) and the boundedness of $\mathcal{M}_{\alpha}$ that
\[
|\left(  \mathcal{F}_{\alpha}f\right)  (x)|=|A_{\alpha}\mathcal{M}_{\alpha
}\mathcal{F}(\mathcal{M}_{\alpha}f)(x\csc\alpha)|\rightarrow 0
\]
as $x\rightarrow\infty$.
\end{proof}

\begin{proposition}
\label{th:cont} The following statements are valid:
\begin{enumerate}
\item[(i)] The FRFT $\mathcal{F}_{\alpha}$ is a bounded linear operator from
$L^{1}(\mathbb{R})\rightarrow L^{\infty}(\mathbb{R})$.

\item[(ii)] For $f\in L^{1}(\mathbb{R})$, $\mathcal{F}_{\alpha}f$ is uniformly
continuous on $\mathbb{R}$.
\end{enumerate}
\end{proposition}

\begin{proof}
(i) It is obvious that $\mathcal{F}_{\alpha}$ is linear. Moreover the claimed boundedness holds as
\[
\left\Vert \mathcal{F}_{\alpha}f\right\Vert _{\infty}     =|A_{\alpha
}|\left\Vert \mathcal{F}(\mathcal{M}_{\alpha}f)\right\Vert _{\infty}
   \leq |A_{\alpha}|\left\Vert \mathcal{M}_{\alpha}f\right\Vert _{1}
  =|A_{\alpha}|\left\Vert f\right\Vert _{1}.
\]
\noindent (ii) For an arbitrary $\varepsilon>0$, it follows from Lemma \ref{th:RL} that there exists $\eta>0$ such that for every $x_i\in \mathbb{R}\setminus [-\eta,\eta]$, $|\left( \mathcal{F}_{\alpha}f\right)  (x_i)|< \varepsilon/2, i=1, 2$. Thus
\[
\left\vert \left(  \mathcal{F}_{\alpha}f\right)  (x_1)-\left(  \mathcal{F}_{\alpha}f\right)  (x_2)\right\vert  < \varepsilon.
\]

For every $x_{1},x_{2}\in [-\eta-1,\eta+1]$,
by the Lagrange mean value theorem, there
exists $\xi$ between $x_{1}$ and $x_{2}$ such that
\begin{align*}
K_{\alpha}(x_{1},t)-K_{\alpha}(x_{2},t)  &  =\frac{\partial}{\partial
x}K_{\alpha}\left(  \xi,t\right)  \left(  x_{1}-x_{2}\right) \\
&  =2\pi iA_{a}\left(  \xi\cot\alpha-t\csc\alpha\right)  K_{\alpha}\left(
\xi,t\right)  \left(  x_{1}-x_{2}\right)  .
\end{align*}
There exist $N>0$ such that for $\left\vert t\right\vert \geq N$,
\[
\int_{\left\vert t\right\vert \geq N}\left\vert f\left(  t\right)  \right\vert
\mathrm{d}t<\frac{\varepsilon}{4}.
\]
Hence,
\begin{align*}
\left\vert \left(  \mathcal{F}_{\alpha}f\right)  \left(  x_{1}\right)
-\left(  \mathcal{F}_{\alpha}f\right)  \left(  x_{2}\right)  \right\vert  &
=\left\vert \int_{-\infty}^{+\infty}\left(  K_{\alpha}(x_{1},t)f(t)-K_{\alpha
}(x_{2},t)f(t)\right)  \mathrm{d}t\right\vert \\
&  \leq2\int_{\left\vert t\right\vert \geq N}\left\vert f\left(  t\right)
\right\vert \mathrm{d}t+\left\vert \int_{\left\vert t\right\vert \leq
N}f(t)\left(  K_{\alpha}(x_{1},t)-K_{\alpha}(x_{2},t)\right)  \mathrm{d}%
t\right\vert \\
&  <\frac{\varepsilon}{2}+2\pi |A_{\alpha}|\int_{\left\vert t\right\vert \leq
N}\left\vert f(t)\right\vert \left\vert \xi\cot\alpha-t\csc\alpha\right\vert
\left\vert x_{1}-x_{2}\right\vert \mathrm{d}t\\
&  <\frac{\varepsilon}{2}+C\left\vert x_{1}-x_{2}\right\vert \int_{\left\vert
t\right\vert \leq N}\left\vert f(t)\right\vert \mathrm{d}t\\
&  \leq\frac{\varepsilon}{2}+C\left\vert x_{1}-x_{2}\right\vert \left\Vert
f\right\Vert _{1},
\end{align*}
where $C$ is a constant independent of $x_1,x_2$. Then for
\[
\left\vert x_{1}-x_{2}\right\vert <\frac{\varepsilon}{2C\left\Vert
f\right\Vert _{1}}.
\]
we obtain
\[
\left\vert \left(  \mathcal{F}_{\alpha}f\right)  \left(  x_{1}\right)
-\left(  \mathcal{F}_{\alpha}f\right)  \left(  x_{2}\right)  \right\vert
<\varepsilon.
\]
So, we conclude that $\mathcal{F}_{\alpha}f$ is uniformly continuous on
$\mathbb{R}$.
\end{proof}

Lemma \ref{th:RL} and Proposition \ref{th:cont} imply that
\begin{equation}\label{implication}
f\in L^{1}(\mathbb{R})\Rightarrow\mathcal{F}_{\alpha}f\in C_{0}(\mathbb{R}).
\end{equation}
A natural question is whether the reverse implication to \eqref{implication} holds, precisely,

\begin{flushleft}
\textbf{Question.} Given $g\in C_{0}(\mathbb{R})$, is there a $L^{1}$-function
$f$ such that $\mathcal{F}_{\alpha}f=g$?
\end{flushleft}

The answer to this question is negative      as the following example illustrates.

\begin{example} Let
\[
g(x)=\left\{
\begin{array}
[c]{ll}%
\left(  \ln x\right)  ^{-1}e^{\pi ix^{2}\cot\alpha}, & x\geq e,\\
xe^{\pi ix^{2}\cot\alpha-1}, & -e<x<e,\\
-\left(  \ln(-x)\right)  ^{-1}e^{\pi ix^{2}\cot\alpha}, & x\leq-e.
\end{array}
\right.
\]
Then $g\in C_{0}(\mathbb{R})$ and $g$ is not the FRFT\ of any $L^{1}%
$-function. To show this, we need first to prove the following.

\begin{claim}
\label{cl:g}If $f\in L^{1}(\mathbb{R})$ and $\mathcal{F}_{\alpha}f$ is odd,
then
\[
\left\vert \int_{\varepsilon}^{N}\frac{(\mathcal{F}_{\alpha}f)(x)}{x}e^{-\pi
ix^{2}\cot\alpha}\mathrm{d}x\right\vert \leq4\left\vert A_{\alpha}\right\vert
\left\Vert f\right\Vert _{1}%
\]
for all $N>\varepsilon>0$.
\end{claim}

Indeed, since $\mathcal{F}_{\alpha}f$ is odd, we have
\begin{align*}
\left(  \mathcal{F}_{\alpha}f\right)  (x)  &  =\frac{1}{2}\left(  \left(
\mathcal{F}_{\alpha}f\right)  (x)-\left(  \mathcal{F}_{\alpha}f\right)
(-x)\right) \\
&  =\frac{1}{2}\int_{-\infty}^{+\infty}f(t)\left(  K_{\alpha}(x,t)-K_{\alpha
}(-x,t)\right)  \mathrm{d}t\\
&  =\frac{A_{\alpha}}{2}\int_{-\infty}^{+\infty}f(t)e^{\pi i\left(
t^{2}+x^{2}\right)  \cot\alpha}\left(  e^{-2\pi ixt\csc\alpha}-e^{2\pi
ixt\csc\alpha}\right)  \mathrm{d}t\\
&  =-iA_{\alpha}e^{\pi ix^{2}\cot\alpha}\int_{-\infty}^{+\infty}f(t)e^{\pi
it^{2}\cot\alpha}\sin\left(  2\pi xt\csc\alpha\right)  \mathrm{d}t.
\end{align*}
Then
\begin{align*}
\int_{\varepsilon}^{N}\frac{(\mathcal{F}_{\alpha}f)(x)}{x}e^{-\pi ix^{2}%
\cot\alpha}\mathrm{d}x  &  =-iA_{\alpha}\int_{\varepsilon}^{N}\frac{1}%
{x}\left(  \int_{-\infty}^{+\infty}f(t)e^{\pi it^{2}\cot\alpha}\sin\left(
2\pi xt\csc\alpha\right)  \mathrm{d}t\right)  \mathrm{d}x\\
&  =-iA_{\alpha}\int_{-\infty}^{+\infty}f(t)e^{\pi it^{2}\cot\alpha}\left(
\int_{2\pi\varepsilon t\csc\alpha}^{2\pi Nt\csc\alpha}\frac{\sin x}%
{x}\mathrm{d}x\right)  \mathrm{d}t.
\end{align*}
Note that
\[
\left|\int_{2\pi\varepsilon t\csc\alpha}^{2\pi Nt\csc\alpha}\frac{\sin x}%
{x}\mathrm{d}x\right|\leq4,\quad\forall~0<\varepsilon<N<+\infty.
\]
Consequently,
\begin{align*}
\left\vert \int_{\varepsilon}^{N}\frac{(\mathcal{F}_{ \alpha}f)(x)}{x}e^{-\pi
ix^{2}\cot\alpha}\mathrm{d}x\right\vert  &  \leq\left\vert A_{\alpha
}\right\vert \int_{-\infty}^{+\infty}\left\vert f(t)\right\vert \left\vert
\int_{2\pi\varepsilon t\csc\alpha}^{2\pi Nt\csc\alpha}\frac{\sin x}%
{x}\mathrm{d}x\right\vert \mathrm{d}t\\
&  \leq4\left\vert A_{\alpha}\right\vert \left\Vert f\right\Vert _{1}.
\end{align*}
So the claim holds. Since $g\in C_{0}(\mathbb{R})$ is an odd function and
\[
\lim_{\substack{t\rightarrow\infty\\\varepsilon\rightarrow0^{+}}}\left\vert
\int_{\varepsilon}^{t}\frac{g(x)}{x}e^{-\pi ix^{2}\cot\alpha}\mathrm{d}%
x\right\vert =\infty,
\]
the above claim implies that $g$ is not the FRFT\ of any $L^{1}$-function.
\end{example}

We conclude this section with a useful identity.

\begin{theorem}
[Multiplication formula]\label{th:mp}For every $f,g\in L^{1}(\mathbb{R})$
and   $\alpha \in \mathbb R$ we have
\begin{equation}
\int_{-\infty}^{+\infty}(\mathcal{F}_{\alpha}f)(x)g(x)\mathrm{d}%
x=\int_{-\infty}^{+\infty}f(x)(\mathcal{F}_{\alpha}g)(x)\mathrm{d}x.
\label{eq:multipl}%
\end{equation}

\end{theorem}

\begin{proof}
The identity (\ref{eq:multipl}) is an immediate consequence of Fubini's
theorem. Indeed,
\begin{align*}
\int_{-\infty}^{+\infty}(\mathcal{F}_{\alpha}f)(x)g(x)\mathrm{d}x  &
=\int_{-\infty}^{+\infty}g(x)\left(  \int_{-\infty}^{+\infty}f(t)K_{\alpha} \left(
x,t\right)  \mathrm{d}t\right)  \mathrm{d}x\\
&  =\int_{-\infty}^{+\infty}f(t)\left(  \int_{-\infty}^{+\infty}%
g(x)K_{\alpha}\left(  x,t\right)  \mathrm{d}x\right)  \mathrm{d}t\\
&  =\int_{-\infty}^{+\infty}f(x)(\mathcal{F}_{\alpha}g)(x)\mathrm{d}x,
\end{align*}
noting that  $K_{\alpha}$  is a bounded function and
$K_{\alpha}\left(  x,t\right)  =K_{\alpha}\left(  t,x\right)  $ for all $x$ and $t$.
\end{proof}

\section{Fractional approximate identities and FRFT inversion on
$L^{1}(\mathbb{R})$}

\label{sect:invs}

In this section, we   study   FRFT inversion. Namely, given the
FRFT of an $L^{1}$-function, how to recover the original function?
We naturally hope that the integral%
\begin{equation}
\int_{-\infty}^{+\infty}(\mathcal{F}_{\alpha}f)(x)K_{-\alpha}(x,t)\mathrm{d}x
\label{eq:in frft}%
\end{equation}
equals $f(t)$. Unfortunately, when $f$ is integrable, one may not necessarily
have that $\mathcal{F}_{\alpha}f$ is integrable, so the integral
(\ref{eq:in frft}) may not make sense.
In fact, $\mathcal{F}_{\frac{\pi}{2}}f$ is  nonintegrable in general (cf.,
\cite[pp. 12]{Duo2001}).

\begin{example}  Let%
\[
f(t)=\left\{
\begin{array}
[c]{cc}%
e^{-\pi\left(  2t+it^{2}\cot\alpha\right)  }, & t\geq0,\\
0, & t<0.
\end{array}
\right.
\]
Then $f\in L^{1}(\mathbb{R})$ but
\[
\left(  \mathcal{F}_{\alpha}f\right)  (x)=\frac{A_{\alpha}e^{\pi ix^{2}%
\cot\alpha}}{2\pi\left(  1+ix\right)  }\notin L^{1}(\mathbb{R}).
\]
\end{example}

To overcome this difficulty, we employ integral summability methods.
We introduce the fractional convolution and
we establish the approximate identities in the fractional setting. Then we
  study the $\Phi_{\alpha}$ means of the fractional Fourier integral,
especially Abel means and Gauss means. Based on the regularity of the
fractional convolution and the results of pointwise convergence, we can
approximate $f$ by the $\Phi_{\alpha}$ means of the integral  \eqref{eq:in frft}.

\subsection{Fractional convolution and approximate identities}

\begin{definition}
\label{def:con}Let $f,g$ be in $L^{1}(\mathbb{R})$. Define the fractional
convolution of order $\alpha$ by%
\[
\left(  f\overset{\alpha}{\ast}g\right)  (x)     =e^{-\pi ix^{2}\cot\alpha
}\int_{-\infty}^{+\infty}e^{\pi it^{2}\cot\alpha}f(t)g(x-t)\mathrm{d}t
   =\mathcal{M}_{-\alpha}\left(  \mathcal{M}_{\alpha}f\ast g\right)  (x).
\]
\end{definition}

We reserve the following notation for the $L^1$ dilation of a function $\phi$
\[
\phi_{\varepsilon}(x):=\frac{1}{\varepsilon}\phi\left(  \frac{x}{\varepsilon
}\right)  ,\quad\forall\varepsilon>0.
\]
The following is a fundamental result concerning fractional convolution and approximate identities.

\begin{theorem}
\label{th:id1}Let $\phi\in L^{1}(\mathbb{R})$ and $\int_{-\infty}^{+\infty
}\phi\left(  x\right)  \mathrm{d}x=1$. If $f\in L^{p}\left(  \mathbb{R}%
\right)  ,1\leq p<\infty$, then%
\[
\lim_{\varepsilon\rightarrow0}\left\Vert \left(  f\overset{\alpha}{\ast}%
\phi_{\varepsilon}\right)  -f\right\Vert _{p}=0.
\]

\end{theorem}

\begin{proof}
Note that
\begin{align*}
\left(  f\overset{\alpha}{\ast}\phi_{\varepsilon}\right)  (x)-f(x)  &
=e^{-\pi ix^{2}\cot\alpha}\int_{-\infty}^{+\infty}e^{\pi it^{2}\cot\alpha
}f(t)\phi_{\varepsilon}(x-t)\mathrm{d}t-\int_{-\infty}^{+\infty}%
\phi_{\varepsilon}\left(  t\right)  f\left(  x\right)  \mathrm{d}t\\
&  =\int_{-\infty}^{+\infty}\left(  e^{\pi i\left(  \left(  x-t\right)
^{2}-x^{2}\right)  \cot\alpha}f(x-t)-f\left(  x\right)  \right)
\phi_{\varepsilon}(t)\mathrm{d}t.
\end{align*}
By Minkowski's integral inequality, we obtain%
\begin{align*}
\left\Vert \left(  f\overset{\alpha}{\ast}\phi_{\varepsilon}\right)
-f\right\Vert _{p}  &  =\left(  \int_{-\infty}^{+\infty}\left\vert
\int_{-\infty}^{+\infty}\left(  e^{\pi i\left(  \left(  x-t\right)  ^{2}%
-x^{2}\right)  \cot\alpha}f(x-t)-f\left(  x\right)  \right)  \phi
_{\varepsilon}(t)\mathrm{d}t\right\vert ^{p}\mathrm{d}x\right)  ^{\frac{1}{p}%
}\\
&  \leq\int_{-\infty}^{+\infty}\left(  \int_{-\infty}^{+\infty}\left\vert
e^{\pi i\left(  \left(  x-t\right)  ^{2}-x^{2}\right)  \cot\alpha
}f(x-t)-f\left(  x\right)  \right\vert ^{p}\mathrm{d}x\right)  ^{\frac{1}{p}%
}\left\vert \phi_{\varepsilon}(t)\right\vert \mathrm{d}t\\
&  =\int_{-\infty}^{+\infty}\left(  \int_{-\infty}^{+\infty}\left\vert e^{\pi
i\left(  \left(  x-\varepsilon t\right)  ^{2}-x^{2}\right)  \cot\alpha
}f(x-\varepsilon t)-f\left(  x\right)  \right\vert ^{p}\mathrm{d}x\right)
^{\frac{1}{p}}\left\vert \phi(t)\right\vert \mathrm{d}t.
\end{align*}

We first prove that
\begin{equation}
J_{\varepsilon}:=\left(  \int_{-\infty}^{+\infty}\left\vert e^{\pi i\left(
\left(  x-\varepsilon t\right)  ^{2}-x^{2}\right)  \cot\alpha}f(x-\varepsilon
t)-f\left(  x\right)  \right\vert ^{p}\mathrm{d}x\right)  ^{\frac{1}{p}%
}\rightarrow0 \label{eq:clm}%
\end{equation}
as $\varepsilon\rightarrow0$.

\bigbreak In fact, for an arbitrary $\eta>0$, since the space of
continuous functions with compact support $C_{c}\left(  \mathbb{R}%
\right)  $ is dense in $L^{p}(\mathbb{R})$, there exists $g\in C_{c}\left(
\mathbb{R}\right)  $ such that%
\[
\left\Vert f-g\right\Vert _{p}<\frac{\eta}{2}.
\]
Since $g$ is uniformly continuous,%
\[
\lim_{\varepsilon\rightarrow0}\left\vert g\left(  x-\varepsilon t\right)
-g\left(  x\right)  \right\vert =0.
\]
Note that%
\begin{align*}
\left\vert J_{\varepsilon}\right\vert \leq &  \left\Vert e^{\pi i\left(
\left(  \left(  \cdot\right)  -\varepsilon t\right)  ^{2}-\left(
\cdot\right)  ^{2}\right)  \cot\alpha}f(\left(  \cdot\right)  -\varepsilon
t)-e^{\pi i\left(  \left(  \left(  \cdot\right)  -\varepsilon t\right)
^{2}-\left(  \cdot\right)  ^{2}\right)  \cot\alpha}g(\left(  \cdot\right)
-\varepsilon t)\right\Vert _{p}\\
&  +\left\Vert e^{\pi i\left(  \left(  \left(  \cdot\right)  -\varepsilon
t\right)  ^{2}-\left(  \cdot\right)  ^{2}\right)  \cot\alpha}g(\left(
\cdot\right)  -\varepsilon t)-g(\left(  \cdot\right)  -\varepsilon
t)\right\Vert _{p}\\
&  +\left\Vert g(\left(  \cdot\right)  -\varepsilon t)-g\right\Vert
_{p}+\left\Vert f-g\right\Vert _{p}\\
=  &  2\left\Vert f-g\right\Vert _{p}+\left\Vert g\right\Vert _{\infty
}\left\Vert e^{\pi i\left(  \left(  \left(  \cdot\right)  -\varepsilon
t\right)  ^{2}-\left(  \cdot\right)  ^{2}\right)  \cot\alpha}-1\right\Vert
_{p}+\left\Vert g(\left(  \cdot\right)  -\varepsilon t)-g\right\Vert _{p}. %
\end{align*}
Consequently, it follows from Lebesgue's dominated convergence theorem that
\begin{align*}
\varlimsup_{\varepsilon\rightarrow0}\left\vert J_{\varepsilon}\right\vert \leq
&  \eta+\varlimsup_{\varepsilon\rightarrow0}\left\Vert g(\left(
\cdot\right)  -\varepsilon t)-g\right\Vert _{p}\\
&  +\left\Vert g\right\Vert _{\infty}\varlimsup_{\varepsilon\rightarrow
0}\left\Vert e^{\pi i\left(  \left(  \left(  \cdot\right)  -\varepsilon
t\right)  ^{2}-\left(  \cdot\right)  ^{2}\right)  \cot\alpha}-1\right\Vert
_{p}=\eta.
\end{align*}
Therefore (\ref{eq:clm}) holds. In view  of
\[
\left(  \int_{-\infty}^{+\infty}\left\vert e^{\pi i\left(  \left(
x-\varepsilon t\right)  ^{2}-x^{2}\right)  \cot\alpha}f(x-\varepsilon
t)-f\left(  x\right)  \right\vert ^{p}\mathrm{d}x\right)  ^{\frac{1}{p}}%
\leq2\left\Vert f\right\Vert _{p}<\infty,
\]
and using Lebesgue's dominated convergence theorem again, we deduce that%
\[
\lim_{\varepsilon\rightarrow0}\left\Vert \left(  f\overset{\alpha}{\ast}%
\phi_{\varepsilon}\right)  -f\right\Vert _{p}=0.
\]

\end{proof}

Next, we discuss the pointwise convergence of approximate identities with
respect to fractional convolution.

\begin{theorem}
\label{th:id2}Let $\phi\in L^{1}(\mathbb{R})$ and $\int_{-\infty}^{+\infty
}\phi\left(  x\right)  \mathrm{d}x=1$. Denote the decreasing radial dominant
functions of $\phi$ by $\psi\left(  x\right)  =\underset{\left\vert
t\right\vert \geq\left\vert x\right\vert }{\sup}\left\vert \phi\left(
t\right)  \right\vert $. If $\psi\in L^{1}(\mathbb{R})$ and $f\in
L^{p}(\mathbb{R}),1\leq p<\infty$, then%
\[
\lim_{\varepsilon\rightarrow0}\left(  f\overset{\alpha}{\ast}\phi
_{\varepsilon}\right)  \left(  x\right)  =f\left(  x\right)  ,\quad
\mathrm{a.e.}\quad x\in\mathbb{R}.
\]

\end{theorem}

\begin{proof}
Since $\psi$ is decreasing and nonnegative, we have
\[
\left\vert x\psi(x)\right\vert \leq2\left\vert \int_{x/2}^{x}\psi
(s)\mathrm{d}s\right\vert \rightarrow0
\]
as $x\rightarrow0$ or $x\rightarrow\infty$. Moreover, there is a constant
$A>0$ such that%
\[
\left\vert x\psi(x)\right\vert \leq A,\quad\forall x\in\mathbb{R}.
\]
As $\mathcal{M}_{\alpha}f\in L^{p}(\mathbb{R})$, it follows from
Lebesgue's differentiation theorem that, for almost all $x\in\mathbb{R}$ we have %
\[
\lim_{r\rightarrow0}\frac{1}{r}\int_{-r}^{r}\left\vert e^{\pi i\left(
x-t\right)  ^{2}\cot\alpha}f\left(  x-t\right)  -e^{\pi ix^{2}\cot\alpha
}f\left(  x\right)  \right\vert dt=0.
\]
Let%
\[
\Omega=\left\{  x:\lim_{r\rightarrow0}\frac{1}{r}\int_{-r}^{r}\left\vert
e^{\pi i\left(  \left(  x-t\right)  ^{2}-x^{2}\right)  \cot\alpha}f\left(
x-t\right)  -f\left(  x\right)  \right\vert dt=0\right\}  ,
\]
and%
\[
G_{x}(t):=\int_{0}^{t}\left\vert e^{\pi i\left(  \left(  x-\tau\right)
^{2}-x^{2}\right)  \cot\alpha}f\left(  x-\tau\right)  -f\left(  x\right)
\right\vert d\tau.
\]
Given $x\in\Omega$ and $\delta>0$, there exists $\eta>0$ such that%
\[
\left\vert \frac{1}{t}G_{x}(t)\right\vert <\delta
\]
whenever $0<\left\vert t\right\vert \leq\eta$. Consider%
\begin{align*}
\left(  f\overset{\alpha}{\ast}\phi_{\varepsilon}\right)  \left(  x\right)
-f\left(  x\right)   &  =\int_{-\infty}^{+\infty}\left(  e^{\pi i\left(
\left(  x-t\right)  ^{2}-x^{2}\right)  \cot\alpha}f(x-t)-f\left(  x\right)
\right)  \phi_{\varepsilon}(t)\mathrm{d}t\\
&  =\left(  \int_{\left\vert t\right\vert \leq\eta}+\int_{\left\vert
t\right\vert \geq\eta}\right)  \left(  e^{\pi i\left(  \left(  x-t\right)
^{2}-x^{2}\right)  \cot\alpha}f(x-t)-f\left(  x\right)  \right)
\phi_{\varepsilon}(t)\mathrm{d}t\\
&  =:I_{1}+I_{2}.
\end{align*}
For $I_{1}$ an integration by parts yields%
\begin{align*}
I_{1}  &  \leq\int_{-\eta}^{\eta}\left\vert e^{\pi i\left(  \left(
x-t\right)  ^{2}-x^{2}\right)  \cot\alpha}f(x-t)-f\left(  x\right)
\right\vert \left\vert \frac{1}{\varepsilon}\phi \Big(\frac{t}{\varepsilon
}\Big)\right\vert \, \mathrm{d}t\\
&  \leq\int_{-\eta}^{\eta}\left\vert e^{\pi i\left(  \left(  x-t\right)
^{2}-x^{2}\right)  \cot\alpha}f(x-t)-f\left(  x\right)  \right\vert \frac
{1}{\varepsilon}\psi\Big(\frac{t}{\varepsilon}\Big)\mathrm{d}t\\
&  =\left.  \frac{1}{t}G(t)\frac{t}{\varepsilon}\psi\Big(\frac{t}{\varepsilon
}\Big)\right\vert _{-\eta}^{\eta}-\int_{-\eta/\varepsilon}^{\eta/\varepsilon}%
\frac{1}{\varepsilon}G(\varepsilon s)\mathrm{d}\psi(s)\\
&  \leq A\delta-\int_{-\eta/\varepsilon}^{\eta/\varepsilon}\frac
{1}{\varepsilon s}G(\varepsilon s)s\mathrm{d}\psi(s)\\
&  \leq A\delta+2\delta\int_{0}^{+\infty}s\mathrm{d}\psi(s)\\
&  \leq A\delta+\left.  2\delta s\psi(s)\right\vert _{0}^{+\infty}+2\delta
\int_{0}^{\infty}\psi(s)\mathrm{d}s\\
&  =A\left(  \delta+2\int_{0}^{\infty}\psi(s)\mathrm{d}s\right)  =:\delta
A_{1}.
\end{align*}
Here, we used that fact that $\psi(x)\geq\left\vert \phi(x)\right\vert $ and
$x\psi(x)\rightarrow0$ as $x\rightarrow0$ or $x\rightarrow\infty$

On the other hand, it follows from H\"{o}lder's inequality that
\begin{align*}
I_{2}  &  \leq\int_{\left\vert t\right\vert \geq\eta}\left\vert e^{\pi
i\left(  \left(  x-t\right)  ^{2}-x^{2}\right)  \cot\alpha}f(x-t)-f\left(
x\right)  \right\vert \left\vert \psi_{\varepsilon}(t)\right\vert
\mathrm{d}t\\
&  \leq\int_{\left\vert t\right\vert \geq\eta}\left\vert f(x-t)\psi
_{\varepsilon}(t)\right\vert \mathrm{d}t-\left\vert f\left(  x\right)
\right\vert \int_{\left\vert t\right\vert \geq\eta}\psi_{\varepsilon
}(t)\mathrm{d}t\\
&  \leq\left\Vert f\right\Vert _{p}\left\Vert \chi_{\eta}\psi_{\varepsilon
}\right\Vert _{p^{\prime}}+\left\vert f\left(  x\right)  \right\vert
\int_{\left\vert t\right\vert \geq\frac{\eta}{\varepsilon}}\psi(t)\mathrm{d}%
t   \rightarrow0.
\end{align*}
as $\varepsilon\rightarrow0$,
where $\chi_{\eta}$ is the characteristic function of the set $\left\{
   x:\,\,   \left\vert x\right\vert \geq\eta\right\}  $. As
 $\delta$ is arbitrary, the theorem is proved.
\end{proof}

\subsection{Fractional Fourier integral   means}

\begin{definition}
Given $\Phi\in C_{0}(\mathbb{R})$ and $\Phi(0)=1$, a function $f$, and
$\varepsilon>0$ we define
\[
M_{\varepsilon,\Phi_{\alpha}}(f):=\int_{-\infty}^{+\infty}(\mathcal{F}_{\alpha}f )(x) K_{-\alpha
}(x,\cdot)\Phi_{\alpha}(\varepsilon x)\mathrm{d}x ,
\]
where
\[
\Phi_{\alpha}\left(  x\right)  :=\Phi\left(  x\csc\alpha\right)   .
\]
The expressions $M_{\varepsilon,\Phi_{\alpha}}(f)$ (with varying $\varepsilon$) are
 called the $\Phi_{\alpha}$ means of the fractional Fourier integral of $f$.
\end{definition}

\begin{theorem}\label{th:phi}
Let $f,\Phi\in L^{1}(\mathbb{R})$. Then for any $\varepsilon>0$ and $t\in \mathbb R$
we have
\begin{align*}
M_{\varepsilon,\Phi_{\alpha}}(f)= f\overset{\alpha}{\ast}\tilde{\varphi}_{\varepsilon},
\end{align*}
where $\varphi:=\mathcal{F}\Phi$ and $\tilde{\varphi}\left(  x\right)  =\varphi\left(
-x\right) $
\end{theorem}

\begin{proof}
Taking advantage of the multiplication formular (\ref{eq:multipl}), we write%
\begin{align*}
M_{\varepsilon,\Phi_{\alpha}}\left(f \right)(t) & =\int_{-\infty}^{+\infty}\left(  \mathcal{F}_{\alpha}f\right)
(x)K_{-\alpha}(x,t)\Phi_{\alpha}\left(  \varepsilon x\right)  \mathrm{d}x\\
&  =A_{-\alpha}e^{-i\pi t^{2}\cot\alpha}\int_{-\infty}^{+\infty}\left(
\mathcal{F}_{\alpha}f\right)  (x)e^{-i\pi x^{2}\cot\alpha}e^{2\pi itx\csc\alpha
}\Phi_{\alpha}\left(  \varepsilon x\right)  \mathrm{d}x\\
&  =A_{-\alpha}A_{\alpha}e^{-i\pi t^{2}\cot\alpha}\!\!\!\int_{-\infty}^{+\infty
}\!\!\!\mathcal{F}\left[  e^{i\pi t^{2}\cot\alpha}f\left(  t\right)  \right]
(x\csc\alpha)e^{2\pi itx\csc\alpha}\Phi\left(  \varepsilon x\csc\alpha\right)
\mathrm{d}x\\
&  =e^{-i\pi t^{2}\cot\alpha}\int_{-\infty}^{+\infty}\mathcal{F}\left[
e^{i\pi t^{2}\cot\alpha}f\left(  t\right)  \right]  (x)e^{2\pi itx}\Phi\left(
\varepsilon x\right)  \mathrm{d}x\\
&  =e^{-i\pi t^{2}\cot\alpha}\int_{-\infty}^{+\infty}e^{i\pi x^{2}\cot\alpha
}f\left(  x\right)  \mathcal{F}\left[  e^{2\pi it(\cdot)}\Phi\left(
\varepsilon(  \cdot)  \right)  \right]  \left(  x\right)
\mathrm{d}x\\
&  =e^{-i\pi t^{2}\cot\alpha}\int_{-\infty}^{+\infty}e^{i\pi x^{2}\cot\alpha
}f\left(  x\right)  \varphi_{\varepsilon}\left(  x-t\right)  \mathrm{d}x\\
&  =\left(  f\overset{\alpha}{\ast}\tilde{\varphi}_{\varepsilon}\right)  (t).
\end{align*}
The desired result is proved.
\end{proof}

In the sequel we will make use of the following well-known results.

\begin{proposition}
[\cite{StW}]\label{pro:poisson}Let $\varepsilon>0$. Then

\begin{enumerate}
\item[(a)] $\mathcal{F}\left[  e^{-2\pi\varepsilon\left\vert\, \cdot\,\right\vert
}\right]  \left(  x\right)  =\frac{1}{\pi}\frac{\varepsilon}{\varepsilon
^{2}+x^{2}}=:P_{\varepsilon}\left(  x\right)  $\quad(Poisson kernel);

\item[(b)] $\mathcal{F}\left[  e^{-4\pi^{2}\varepsilon\left| \, \cdot\,\right|
^{2}}\right]  \left(  x\right)  =\frac{1}{\left(  4\pi\varepsilon\right)
^{1/2}}e^{-x^{2}/4\varepsilon}=:W_\varepsilon\left(  x\right)  $%
\quad(Weierstrass kernel).
\end{enumerate}
\end{proposition}

\begin{lemma}
[\cite{Duo2001}]\label{lm1} For every $\varepsilon>0$, the Weierstrass and
Poisson kernels satisfy

\begin{enumerate}
[(i)]

\item $W_\varepsilon,P_{\varepsilon}\in L^{1}(\mathbb R)$;

\item $\int_{-\infty}^{+\infty}W_\varepsilon(x)\, \mathrm{d}x=\int_{-\infty
}^{+\infty}P_{\varepsilon}(x)\, \mathrm{d}x=1$.
\end{enumerate}
\end{lemma}

\begin{definition}
\label{def:PGW}For $f\in L^{p}(\mathbb{R})$, $1\leq p<\infty$, and
$\varepsilon>0$, the expressions
\[
u_{\alpha}\left(  t,\varepsilon\right)  :=\left(  f\overset{\alpha}{\ast
}\tilde{P}_{\varepsilon}\right)  (t)=\mathcal{M}_{\alpha}\left[ \int_{-\infty
}^{+\infty}\mathcal{M}_{\alpha}f(x)P_{\varepsilon}( (\cdot)-x)\, \mathrm{d}x\right] (t)
\]
are called the  fractional Poisson integrals of $f$.  The expressions
\[
S_{\alpha}\left(  t,\varepsilon\right)  :=\left(  f\overset{\alpha}{\ast
}\tilde{W}_\varepsilon\right)  (t)=\mathcal{M}_{\alpha
} \left[ \int_{-\infty}^{+\infty}\mathcal{M}_{\alpha}f(x)W_\varepsilon((\cdot)-x)\mathrm{d}x \right](t)
\]
are called   and   fractional
Gauss-Weierstrass integrals  of $f$.
\end{definition}

We now focus on two  functions that give rise to
special $\Phi_\alpha$ means. Denote by
\[p_{\alpha}\left(  x\right)  =e^{-2\pi\varepsilon\left\vert \csc
\alpha\right\vert \left\vert x\right\vert }\quad \text{and} \quad w_{\alpha}\left(  x\right)= e^{-4\pi^{2}\varepsilon
x^{2}\csc^{2}\alpha}.
\]

\begin{definition}
 The $\Phi_{\alpha}$ means%
\[
M_{\varepsilon,p_\alpha}(f) =\int_{-\infty}^{+\infty}\left(  \mathcal{F}_{\alpha}f\right)  (x)
K_{-\alpha}(x,\cdot)e^{-2\pi\varepsilon\left\vert \csc\alpha\right\vert \left\vert
x\right\vert }\, \mathrm{d}x
\]
are called the \emph{Abel means of the fractional Fourier
integral of $f$}, while
\[
M_{\varepsilon,w_\alpha}(f) =\int_{-\infty}^{+\infty}\left(  \mathcal{F}_{\alpha}f\right)  (x)K_{-\alpha
}(x,\cdot)e^{-4\pi^{2}\varepsilon^{2}x^{2}\csc^{2}\alpha}\mathrm{d}x
\]
are called the \emph{Gauss means  of the fractional Fourier
integral of $f$}.
\end{definition}

By Theorem \ref{th:phi} and Proposition \ref{pro:poisson}, the Poisson
integrals and Gauss-Weierstrass integrals of $f$ are the Abel and Gauss means, respectively.
It is straightforward to verify the following identities.

\begin{proposition}
\label{th:PGW}If $f\in L^{1}(\mathbb{R})$, then for any $\varepsilon>0$, the
following identities are valid
\begin{enumerate}
\item[(a)] $u_{\alpha}\left(  t,\varepsilon\right)
   =M_{\varepsilon,p_\alpha}(f)(t)$;
\item[(b)] $S_{\alpha} (  t,\varepsilon^{2} )=M_{\varepsilon,w_\alpha}(f) (t)$.
\end{enumerate}
\end{proposition}

\subsection{FRFT inversion}

We now address the FRFT inversion problem.
In view of Theorems~\ref{th:id1}, ~\ref{th:id2} and
~\ref{th:phi}, we can derive the following   conclusions.

\begin{theorem}
\label{cor:L1}If $\Phi,\varphi:=\mathcal{F}\Phi\in L^{1}(\mathbb{R})$ and
$\int_{-\infty}^{+\infty}\varphi\left(  x\right)  dx=1$, then the
$\Phi_{\alpha}$ means of the Fourier integral of $f$ are convergent to
$f$ in the sense of $L^{1}$ norm, that is,
\[
\lim_{\varepsilon\rightarrow0}\left\Vert \int_{-\infty}^{+\infty}\left(
\mathcal{F}_{\alpha}f\right)  (x)K_{-\alpha}(\cdot,x)\Phi_{\alpha}\left(
\varepsilon x\right)  \mathrm{d}x-f\left(  \cdot\right)  \right\Vert _{1}=0.
\]

\end{theorem}

\begin{theorem}
\label{cor:point}If $\Phi,\varphi:=\mathcal{F}\Phi\in L^{1}(\mathbb{R})$,
$\psi=\underset{\left\vert t\right\vert \geq\left\vert x\right\vert }{\sup
}\left\vert \varphi\left(  t\right)  \right\vert \in L^{1}(\mathbb{R})$ and
$\int_{-\infty}^{+\infty}\varphi\left(  x\right)  dx=1$, then the
$\Phi_{\alpha}$ means of the Fourier integral of $f$ are a.e. convergent to
$f$,   that is,
\[
\int_{-\infty}^{+\infty}\left(  \mathcal{F}_{\alpha}f\right)  (x)K_{-\alpha
}(t,x )\Phi_{\alpha}\left(  \varepsilon x\right)  \mathrm{d}x\rightarrow
f\left(  t\right)
\]
 as $\varepsilon\rightarrow0$ for almost all $t\in\mathbb{R}$.
\end{theorem}

\bigbreak In particular, in view of Theorem \ref{cor:L1}-\ref{cor:point}, Proposition \ref{th:PGW} and the properties of Weierstrass kernel and Poisson kernel
(Lemma \ref{lm1}), we deduce the following result.

\begin{corollary}
\label{cor:p-w}If $f\in L^{1}(\mathbb{R})$, then the Gauss and Abel means of
the fractional Fourier integral of $f$   converge
to $f$ in $L^{1}$   and a.e., that is,
\[
\lim\limits_{\varepsilon\rightarrow0}\left\Vert M_{\varepsilon,p_\alpha}(f)-f  \right\Vert _{1}=0 , \quad \lim\limits_{\varepsilon\rightarrow0}\left\Vert M_{\varepsilon,w_\alpha}(f)-f
\right\Vert _{1}=0,\]
and
\[
M_{\varepsilon,p_\alpha}(f)(t)\to f(t),\quad M_{\varepsilon,w_\alpha}(f)(t)\to f(t)
\]
for almost all $  t \in\mathbb{R} $ as $\varepsilon\to 0$.
\end{corollary}

\begin{corollary}
\label{th:invers}If $f,\mathcal{F}_{\alpha}f\in L^{1}$, then for almost all
$x\in\mathbb{R}$, we have%
\[
f(t)=\int_{-\infty}^{+\infty}(\mathcal{F}_{\alpha}f)(x)K_{-\alpha}(x,t)\, \mathrm{d}x.
\]
\end{corollary}

\begin{proof}
Consider the Gauss mean of the fractional Fourier integral $\mathcal{F}%
_{\alpha}f$. On one hand, it follows from Corollary \ref{cor:p-w} that%
\[
M_{\varepsilon,w_\alpha}(f)(t)=\int_{-\infty}^{+\infty}(\mathcal{F}_{\alpha}f)(x)K_{-\alpha}(x,t)e^{-4\pi
^{2}\varepsilon x^{2}\csc^{2}\alpha}\mathrm{d}x\rightarrow f\left(  t\right)
\]
for almost all $t\in\mathbb{R}$, as $\varepsilon\rightarrow0$.
On the other hand, as $\mathcal{F}_{\alpha}f\in L^{1}(\mathbb{R})$, by the
Lebesgue dominated convergence theorem we obtain that
\[
\int_{-\infty}^{+\infty}\left(  \mathcal{F}_{\alpha}f\right)  (x)K_{-\alpha
}(x,t)e^{-4\pi^{2}\varepsilon x^{2}\csc^{2}\alpha}\mathrm{d}x\rightarrow
\int_{-\infty}^{+\infty}\left(  \mathcal{F}_{\alpha}f\right)  (x)K_{-\alpha
}(x,t)\, \mathrm{d}x
\]
as $\varepsilon\rightarrow0$. This proves the desired result.
\end{proof}

\begin{corollary}
Let $f\in L^{1}(\mathbb{R})$. If $\mathcal{F}_{\alpha}f\geq0$ and $f$ is
continuous at $t=0$, then $\mathcal{F}_{\alpha}f\in L^{1}(\mathbb{R})$.
Furthermore,
\[
f(t)=\int_{-\infty}^{+\infty}(\mathcal{F}_{\alpha}f)(x)K_{-\alpha
}(x,t)\, \mathrm{d}x,\quad\text{for almost all }t\in\mathbb{R}.
\]
In particular, $\int_{-\infty}^{+\infty}\left(  \mathcal{F}_{\alpha}f\right)
(x)\, \mathrm{d}x=f\left(  0\right)  $.
\end{corollary}

\begin{remark}
\noindent (i)  Even if $\mathcal{F}_{\alpha}f\notin L^{1}(\mathbb{R})$, the Gauss and
Abel means of the integral
\[
\int_{-\infty}^{+\infty}(\mathcal{F}_{\alpha}f)(x)K_{-\alpha}(x,t)\,  \mathrm{d}x
\]
may make sense. For example, if $\mathcal{F}_{\alpha}f\notin L^{1}%
(\mathbb{R})$ and $\mathcal{F}_{\alpha}f$ is bounded, then
\[
M_{\varepsilon,p_\alpha}(f)(t) ,M_{\varepsilon,w_\alpha}(f)(t) <\infty \quad\forall\varepsilon>0.
\]
\noindent (ii)
 Even if $\mathcal{F}_{\alpha}f\notin L^{1}(\mathbb{R})$, the limits
$ \lim\limits_{\varepsilon\rightarrow 0} u_{\alpha}\left(  t,\varepsilon
\right)  \ $and $\lim\limits_{\varepsilon\rightarrow 0} S_{\alpha} (
t,\varepsilon^{2} )  $ may exist. For example, this is the case when $(\mathcal{F}_{\alpha
}f)(x)=\sin x/x$.
\end{remark}

\begin{theorem}
[Uniqueness of FRFT on $L^{1}(\mathbb{R})$]If $f_{1},f_{2}\in L^{1}%
(\mathbb{R})$ and $\left(  \mathcal{F}_{\alpha}f_{1}\right)  (x)=\left(
\mathcal{F}_{\alpha}f_{2}\right)  (x)$ for all $x\in\mathbb{R}$, then
\begin{equation}
f_{1}\left(  t\right)  =f_{2}\left(  t\right)  ,\quad\text{a.e. }%
t\in\mathbb{R}. \label{eq:unq}%
\end{equation}

\end{theorem}

\begin{proof}
Let $g=f_{1}-f_{2}$. Then
\[
\mathcal{F}_{\alpha}g=\mathcal{F}_{\alpha}f_{1}-\mathcal{F}_{\alpha}f_{2}.
\]
It follows from Corollary \ref{th:invers} that%
\[
g\left(  x\right)  =\int_{-\infty}^{+\infty}(\mathcal{F}_{\alpha
}g)(t)K_{-\alpha}(x,t)\, \mathrm{d}t=0
\]
a.e. on $\mathbb{R}$, which implies ($\ref{eq:unq}$).
\end{proof}

\section{FRFT on $L^{p}(\mathbb{R})~~(1<p<2)$}

\label{sect:Lp}

Having set down the basic facts concerning the action of the FRFT on
$L^{1}(\mathbb{R})\ $and $L^{2}(\mathbb{R})$, we now     extend its
definition on $L^{p}(\mathbb{R})$ for $1<p<2$.
Note that $L^{p}(\mathbb{R})$ is contained in $ L^{1}(\mathbb{R})+L^{2}(\mathbb{R})$ for
$1<p<2$, where
\[
L^{1}(\mathbb{R})+L^{2}(\mathbb{R})=\left\{
 f_{1}+f_{2}:\,\, f_{1}\in L^{1}(\mathbb{R}),f_{2}\in L^{2}(\mathbb{R})\right\}.
\]

\begin{definition}
For $f\in L^{p}(\mathbb{R})$, $1<p<2$, with
\[
f=f_{1}+f_{2},\qquad f_{1}\in L^{1}(\mathbb{R}), f_{2}\in L^{2}(\mathbb{R}),
\]
the FRFT of order $\alpha$ of $f$ defined by $\mathcal{F}_{\alpha
}f=\mathcal{F}_{\alpha}f_{1}+\mathcal{F}_{\alpha}f_{2}$.
\end{definition}

\begin{remark}
The decomposition  of $f$ as $f_1+f_2$ is not unique. However, the definition of
$\mathcal{F}_{\alpha}f$ is independent on the choice of $f_{1}$ and $f_{2}$.
If $f_{1}+f_{2}=g_{1}+g_{2}$ for $f_{1},g_{1}\in L^{1}(\mathbb{R})$ and
$f_{2},g_{2}\in L^{2}(\mathbb{R})$, we have $f_{1}-h_{1}=f_{2}-h_{2}\in
L^{1}(\mathbb{R})\cap L^{2}(\mathbb{R})$. Since those functions are equal,
their FRFT are also equal, and we
obtain $\mathcal{F}_{\alpha}f_{1}-\mathcal{F}_{\alpha}h_{1}=\mathcal{F}%
_{\alpha}f_{2}-\mathcal{F}_{\alpha}h_{2}$, using the linearity of the
FRFT, which yields $\mathcal{F}_{\alpha }(f_{1}+f_{2})=\mathcal{F}_{\alpha}(h_{1}+h_{2})$.
\end{remark}

We have the following result concerning the action of the FRFT\ on
$L^{p}(\mathbb{R})$.

\begin{theorem}
[Hausdorff-Young inequality]Let $1<p\leq2$, $p^{\prime}=p/(p-1)$. Then $\mathcal{F}_{\alpha}$
are bounded linear operators from $L^{p}(\mathbb{R})$ to $L^{p^{\prime}%
}(\mathbb{R})$. Moreover,
\begin{equation}\label{HYineq}
\left\Vert \mathcal{F}_{\alpha}f\right\Vert _{p^{\prime}}\leq A_\alpha^{\frac2p-1}\left\Vert
f\right\Vert _{p}.
\end{equation}

\end{theorem}

\begin{proof}
By Proposition~\ref{th:cont} (i)  $\mathcal{F}%
_{\alpha}$ maps $L^1$ to $L^\infty$ (with norm bounded by $A_\alpha$) and Theorem \ref{prop:L2} (ii), it maps
$L^2$ to $L^2$ with (with norm $1$).
It follows from the   Riesz-Thorin interpolation theorem
Hausdorff-Young inequality \eqref{HYineq} holds.
\end{proof}

FRFT inversion also holds on $L^{p}(\mathbb{R})$ ($1<p<2$) and this can be proved by an
argument similar to that for  $L^{1}(\mathbb{R})$
via the use of Theorems \ref{th:id1}-\ref{th:id2}. We won't go into much detail here.

\section{Multiplier theory and Littlewood-Paley theorem associated with the FRFT}
\label{sect:multiplier}
\subsection{Fractional Fourier transform multipliers}

 Fourier multipliers play an important role in
operator theory, partial differential equations, and harmonic analysis.
In this section, we   study some basic multiplier theory results in the
FRFT context.

\begin{definition}
\label{def:ma}Let $1\leq p\leq\infty$ and $m_{\alpha}\in L^{\infty}%
(\mathbb{R})$. Define the operator $T_{m_{\alpha}}$ as%
\[
\mathcal{F}_{\alpha}\left(  T_{m_{\alpha}}f\right)  \left(  x\right)
=m_{\alpha}\left(  x\right)  \left(  \mathcal{F}_{\alpha}f\right)  \left(
x\right)  ,\quad\forall f\in L^{2}(\mathbb{R})\cap L^{p}(\mathbb{R}).
\]
The function $m_{\alpha}$ is called \emph{the }$L^{p}$\emph{ Fourier
multiplier of order }$\alpha$, if there exist a constant $C_{p,\alpha}>0$ such
that
\begin{equation}
\left\Vert T_{m_{\alpha}}f\right\Vert _{p}\leq C_{p,\alpha}\left\Vert
f\right\Vert _{p},\quad\forall f\in L^{2}(\mathbb{R})\cap L^{p}(\mathbb{R}).
\label{eq:ma}%
\end{equation}
\end{definition}

\noindent As $L^{2}(\mathbb{R})\cap L^{p}(\mathbb{R})$ is dense in $L^{p}%
(\mathbb{R})$, there is a unique bounded extension of $T_{m_{\alpha}}$  in
$L^{p}(\mathbb{R})$ satisfying (\ref{eq:ma}). This extension is also denoted by $T_{m_{\alpha}}$. Define
\[
\left\Vert m_{\alpha}\right\Vert =\left\Vert T_{m_{\alpha}}\right\Vert
_{(p.p)}:=\sup_{\substack{f\in L^{2}\cap L^{p}\\\left\Vert f\right\Vert
_{p}\leq1}}\left\Vert \mathcal{F}_{-\alpha}\left[  m_{\alpha}\left(
\mathcal{F}_{\alpha}f\right)  \right]  \right\Vert _{p}.
\]

In view of Definition \ref{def:ma}, many important fractional integral
operators can be expressed in terms of fractional $L^{p}$ multiplier.

\begin{figure}[htb]
\centering
\begin{tikzpicture}[>=stealth,thick]

\draw [->,dashed,very thick, draw=blue!70!black,rotate around y=55] (-2.5,0.5,0) -- (2.5,-0.5,0) node [at end, right] {$\omega$};
\draw [->,very thick] (0,-3,0) -- (0,3,0) node [at end, left] {Im $t$};
\draw [->,very thick] (0,0,-3) -- (0,0,3) node [at end, left] {Re $t$};

\begin{scope}[z={(-135:7.4mm)},x={(-10:8.4mm)}]
\draw [fill=blue!20,draw=blue!70!red]  (1,0,0) -- (1,0,1) -- (2.5,0,1)--(2.5,0,0);
\draw [draw=blue!50!green,fill=blue!50!green!20] (-1,0,0) -- (-1,0,1) -- (-2.5,0,1)--(-2.5,0,0);
\draw [->,,very thick] (-3,0,0) -- (3,0,0) node [at end, right] {$\omega'$};

\begin{scope}[canvas is zx plane at y=0]
\draw[->,thick,red] (0,1.5) arc (80:170:0.5);
\end{scope}

\end{scope}

\node[red] at (1.3,0,0)[right] {\tiny $\frac{\pi}{2}-\alpha$};

\node at (0,-3,0)[below]{\footnotesize (a) the original signal: $\mathcal F_\alpha (u)(\omega')$};

\node[draw=red!50!black,fill=red!20,minimum width=0.2mm, minimum height=8mm, single arrow,single arrow head extend=1.5mm, single arrow head indent=.4mm] at (4.2,0,0) {$\quad$};
\end{tikzpicture}~~
\begin{tikzpicture}[>=stealth,thick]
\draw [->,dashed,very thick, draw=blue!70!black,rotate around y=55] (-2.5,0.5,0) -- (2.5,-0.5,0) node [at end, right] {$\omega$};
\draw [->,very thick] (0,-3,0) -- (0,3,0) node [at end, left] {Im $t$};
\draw [->,very thick] (0,0,-3) -- (0,0,3) node [at end, left] {Re $t$};
\begin{scope}[z={(-135:5mm)},x={(-10:8.4mm)},y={(90:9mm)}]
\draw [fill=blue!20,draw=blue!70!red,rotate around x=90]  (1,0,0) -- (1,0,1) -- (2.5,0,1)--(2.5,0,0);
\draw [draw=blue!50!green,fill=blue!50!green!20,rotate around x=-90] (-1,0,0) -- (-1,0,1) -- (-2.5,0,1)--(-2.5,0,0);
\draw [->,very thick] (-3,0,0) -- (3,0,0) node [at end, right] {$\omega'$};

\begin{scope}[canvas is zx plane at y=0]
\draw[->,thick,red] (0,1.5) arc (75:160:0.7);
\end{scope}

\end{scope}

\node[red] at (1.3,0,0)[right] {\tiny $\frac{\pi}{2}-\alpha$};

\node at (0,-3,0)[below]{\footnotesize (b) after Hilbert transform: $\mathcal F_\alpha (\mathcal H_{\alpha}u)(\omega')$};
\end{tikzpicture}

\caption{Hilbert transform of order $\alpha$ in frequency domain}%
\label{fig:HT}%
\end{figure}
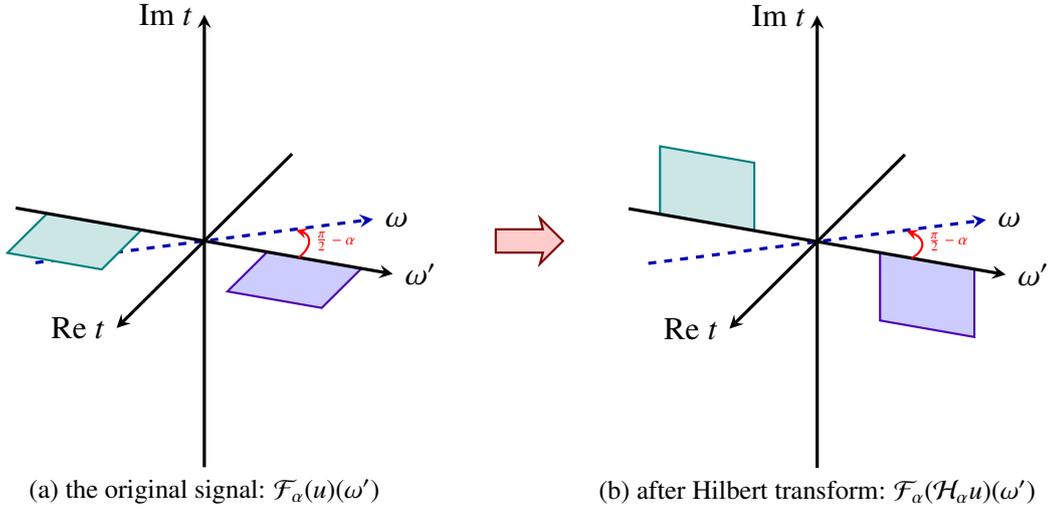

\begin{example}
Recall that the classical Hilbert transform is defined as
\begin{equation}
(\mathcal H u)\left(  \omega\right)  = \mathrm{p.v.~}\frac{1}{\pi} \int_{-\infty}^{+\infty}\frac{u(t)}%
{\omega -t}\mathrm{d}t.\label{eq:CHT}%
\end{equation}
The Hilbert transform of order $\alpha$ is defined as (cf., \cite{zayed})%
\begin{equation}
(\mathcal H_{\alpha}u)\left(  \omega'\right)  =\mathrm{p.v.~} e^{-i\pi\omega'^2%
\cot\alpha}  \frac{1}{\pi} \int_{-\infty}^{+\infty}\frac{u(t)e^{i\pi t^{2}\cot\alpha}}%
{\omega' -t}\mathrm{d}t.\label{eq:HT}%
\end{equation}
For $1<p<\infty$, the operator $\mathcal H_{\alpha}$ is bounded from $L^{p}%
(\mathbb{R})$ to $L^{p}(\mathbb{R})$. By \cite[Theorem 4]{zayed}, we see that
$m_{\alpha}=-i\mathrm{sgn}\left(  (\pi-\alpha)\omega'\right)  $ is a
fractional $L^{p}$ multiplier and the associated operator $T_{m_{\alpha}}$ is
the fractional Hilbert transform, that is,%
\begin{equation}
(\mathcal{F}_{\alpha}\mathcal H_{\alpha}u)\left(  \omega'\right)
=-i\mathrm{sgn}\left(  (\pi-\alpha)\omega'\right)  \left(
\mathcal{F}_{\alpha}u\right)  \left(  \omega'\right)  .\label{eq:mHT}%
\end{equation}
Without loss of generality, assume that $\alpha\in(0,\pi)$. It can be seen from
(\ref{eq:mHT}) that the Hilbert transform of order $\alpha$ is a phase-shift
converter that multiplies the positive frequency portion of the original
signal by $-i$, that is, maintaining the same amplitude, shifts the phase by
$-\pi/2$, while the negative frequency portion is shifted by $\pi/2$. As
shown in Fig. \ref{fig:HT}.
\end{example}

\begin{example}
Let $m_{\alpha}=e^{-2\pi\varepsilon\left\vert \csc\alpha\right\vert \left\vert
x\right\vert }$. Then the corresponding operator $T_{m_{\alpha}}$ is the fractional
Poisson integral (see Definition \ref{def:PGW}). In view of the Young's
inequality and Lamma \ref{lm1}, we know that $m_{\alpha}$ is a fractional
$L^{p}$ multiplier for $1\leq p<\infty$. Similarly, the fractional
Gauss-Weierstrass integral is the operator $T_{m_{a}}$ associated with the
fractional $L^{p}$ multiplier $m_{\alpha}=e^{-4\pi^{2}\varepsilon x^{2}%
\csc^{2}\alpha}$.
\end{example}

\begin{example}
Let $a,b\in\mathbb{R}$ and $a<b$. Denote the characteristic function of the
interval $[a,b]$ by $\chi_{\lbrack a,b]}$. Later, in the proof of
Littlewood-Paley theorem (Theorem \ref{th:L-P}), equality (\ref{eq:L-P1}) will
show that $\chi_{[ a,b]}$ is a $L^{p}$ ($1<p<\infty$) multiplier in the
FRFT context. The associated operator $T_{\chi_{[ a,b]}}$ acting on a signal $u$ is equivalent to making a truncation in the frequency domain of the original signal.
\end{example}

 The following theorem provides a sufficient condition for
judging $L^{p}$ multiplier, which is the H\"{o}rmander-Mikhlin multiplier
theorem in the fractional setting.

\begin{theorem}
Let $m_{\alpha}$ be a bounded function. If there exists a constant $B>0$ such
that one of the following condition holds:

\begin{enumerate}
[(a)]

\item (Mikhlin's condition)%
\begin{equation}
\left\vert \frac{\mathrm{d}}{\mathrm{d}x}m_{\alpha}(x)\right\vert \leq
B\left\vert x\right\vert ^{-1}; \label{eq:M}%
\end{equation}

\item (H\"{o}rmander's condition)
\begin{equation}
\sup_{R>0}\frac{1}{R}\int_{R<\left\vert x\right\vert <2R}\left\vert
\frac{\mathrm{d}}{\mathrm{d}x}m_{\alpha}(x)\right\vert ^{2}\mathrm{d}x\leq
B^{2}. \label{eq:MH}%
\end{equation}

\end{enumerate}
Then $m_{\alpha}$ is a fractional $L^{p}$ multiplier for $1<p<\infty$, that
is, there exist a constant $C>0$ such that
\[
\left\Vert T_{m_{\alpha}}f\right\Vert _{p}=\left\Vert \mathcal{F}_{-\alpha
}\left[  m_{\alpha}\left(  \mathcal{F}_{\alpha}f\right)  \right]  \right\Vert
_{p}\leq C\left\Vert f\right\Vert _{p},\quad\forall f\in L^{p}(\mathbb{R}).
\]

\end{theorem}

\begin{proof}
In view of the decomposition (\ref{eq:def}), we have%
\[
\mathcal{F}_{\alpha}\left(  T_{m_{\alpha}}f\right)  \left(  x\right)
=A_{\alpha}e^{i\pi x^{2}\cot\alpha}\mathcal{F}[e^{i\pi(\cdot)^{2}\cot\alpha
}\left(  T_{m_{\alpha}}f\right)  ](x\csc\alpha)
\]
and%
\[
m_{\alpha}(x)\left(  \mathcal{F}_{\alpha}f\right)  (x)=m_{\alpha}(x)A_{\alpha
}e^{i\pi x^{2}\cot\alpha}\mathcal{F}[e^{i\pi(\cdot)^{2}\cot\alpha}%
f](x\csc\alpha).
\]
Then%
\[
\mathcal{F}[e^{i\pi(\cdot)^{2}\cot\alpha}\left(  T_{m_{\alpha}}f\right)
](x)=\tilde{m}_{\alpha}(x)\mathcal{F}[e^{i\pi(\cdot)^{2}\cot\alpha}f](x),
\]
where $\tilde{m}_{\alpha}(x)=m_{\alpha}(x\sin\alpha)$. Namely,
\[
T_{m_{\alpha}}f=e^{-i\pi(\cdot)^{2}\cot\alpha}\mathcal{F}^{-1}\left[  \tilde
{m}_{\alpha}\mathcal{F}\left(  e^{i\pi(\cdot)^{2}\cot\alpha}f\right)  \right]
,\quad\forall f\in L^{p}(\mathbb{R}).
\]

It is obvious that $\tilde{m}_{\alpha}$ satisfies (\ref{eq:M}) or
(\ref{eq:MH}) and $g:=e^{i\pi(\cdot)^{2}\cot\alpha}f\in L^{p}(\mathbb{R})$.
Therefore, it follows from the classical H\"{o}rmander-Mihlin multiplier
theorem (cf., \cite{Hor,mih,GL}) that $\tilde{m}_{\alpha} $ is an $L^p$
Fourier multiplier.
Consequently,%
\begin{align*}
\left\Vert T_{m_{\alpha}}f\right\Vert _{p}  &  =\left\Vert e^{-i\pi(\cdot
)^{2}\cot\alpha}\mathcal{F}^{-1}\left[  \tilde{m}_{\alpha}\mathcal{F}\left(
e^{i\pi(\cdot)^{2}\cot\alpha}f\right)  \right]  \right\Vert _{p}\\
&  =\left\Vert \mathcal{F}^{-1}\left[  \tilde{m}_{\alpha}(\mathcal{F}%
g)\right]  \right\Vert _{p}\\
&  \leq C\left\Vert g\right\Vert _{p}=C\left\Vert f\right\Vert _{p}%
\end{align*}
for some positive constant $C$.
\end{proof}

The proof of the following two FRFT multiplier theorems is obtained in
a similar fashion.

\begin{theorem}
[Bernstein multiplier theorem]Let $m_{\alpha}\in C^{1}(\mathbb{R}%
\backslash\{0\})$ be bounded. If $\Vert m_{\alpha}^{\prime}\Vert<\infty$, then
there exist constants $C_{1},C_{2}>0$ such that
\[
\left\Vert \mathcal{F}_{-\alpha}\left[  m_{\alpha}\left(  \mathcal{F}_{\alpha
}f\right)  \right]  \right\Vert _{p}\leq C_{1}\left\Vert f\right\Vert _{p},
\]
for $f\in L^{p}(\mathbb{R})$ ($1\leq p<\infty$), and%
\[
\left\Vert m_{\alpha}\right\Vert ^{2}\leq C_{2}\Vert m_{\alpha}\Vert_{2}\Vert
m_{\alpha}^{\prime}\Vert_{2}.
\]

\end{theorem}

\begin{theorem}
[Marcinkiewicz multiplier theorem]Let $m_{\alpha}\in L^{\infty}(\mathbb{R}%
)\cap C^{1}(\mathbb{R}\backslash\{0\})$. If there exist a constant $B>0$ such
that
\[
\sup_{I\in\mathcal{I}}\int_{I}\left\vert \frac{\mathrm{d}}{\mathrm{d}%
x}m_{\alpha}(x)\right\vert \mathrm{d}x\leq B,
\]
where $\mathcal{I}:\mathcal{=}\{[2^{j},2^{j+1}],[-2^{j+1},-2^{j}%
]\}_{j\in\mathbb{Z}}$ is the set of binary intervals in $\mathbb{R}$, then,
for $f\in L^{p}(\mathbb{R})$ ($1<p<\infty$), there exist a constant $C>0$ such
that
\[
\left\Vert \mathcal{F}_{-\alpha}\left[  m_{\alpha}\left(  \mathcal{F}_{\alpha
}f\right)  \right]  \right\Vert _{p}\leq C\left\Vert f\right\Vert _{p}.
\]

\end{theorem}

\subsection{Littlewood-Paley theorem in the FRFT context}

In this subsection we
study the Littlewood-Paley theorem in the FRFT context. The Littlewood-Paley
is not only a powerful tool in Fourier analysis, but also plays an important
role in other areas, such as  partial
differential equations.

Let $j\in\mathbb{Z}$. Define the binary intervals in $\mathbb{R}$ as%
\[
\left\{
\begin{array}
[c]{cc}%
I_{j}^{\alpha}:=[2^{j}\sin\alpha,2^{j+1}\sin\alpha],-I_{j}^{\alpha}%
:=[-2^{j+1}\sin\alpha,-2^{j}\sin\alpha], & \alpha\in(0,\pi),\\
I_{j}^{\alpha}:=[2^{j+1}\sin\alpha,2^{j}\sin\alpha],-I_{j}^{\alpha}%
:=[-2^{j}\sin\alpha,-2^{j+1}\sin\alpha], & \alpha\in(\pi,2\pi).
\end{array}
\right.
\]
Then those binary intervals$\mathbb{\ }$internally disjoint and%
\[
\mathbb{R}\backslash\{0\}=\bigcup\limits_{j\in\mathbb{Z}}(-I_{j}^{\alpha}\cup
I_{j}^{\alpha}).
\]
Let $\mathcal{I}_{\alpha}:=\{I_{j}^{\alpha},-I_{j}^{\alpha}\}_{j\in\mathbb{Z}%
}$. Define the partial summation operator $S_{\rho_{\alpha}}$ corresponding to
$\rho_{\alpha}\in\mathcal{I}_{\alpha}$ by%
\[
\mathcal{F}_{\alpha}(S_{\rho_{\alpha}}f)(x)=\chi_{\rho_{\alpha}}(x)\left(
\mathcal{F}_{\alpha}f\right)  (x),\quad\forall f\in L^{2}(\mathbb{R})\cap
L^{p}(\mathbb{R}),
\]
where $\chi_{\rho_{\alpha}}$ denote the characteristic function of the
interval $\rho_{\alpha}$. It is obvious that%

\begin{equation}
\sum_{\rho_{\alpha}\in\mathcal{I}_{\alpha}}\left\Vert S_{\rho_{\alpha}%
}(f)\right\Vert _{2}^{2}=\left\Vert f\right\Vert _{2}^{2},\quad\forall f\in
L^{2}(\mathbb{R}). \label{eq:Sp}%
\end{equation}
For general $L^{p}(\mathbb{R})$ functions, we have the following result, which
is the Littlewood-Paley theorem in fractional setting.

\begin{theorem}
\label{th:L-P}Let $f\in L^{p}(\mathbb{R})$, $1<p<\infty$. Then%
\[
\left(  \sum_{\rho_{\alpha}\in\mathcal{I}_{\alpha}}\left\vert S_{\rho_{\alpha
}}(f)\right\vert ^{2}\right)  ^{1/2}\in L^{p}(\mathbb{R})
\]
and there exists constants $C_{1},C_{2}>0$ independent of $f$ such that%
\[
C_{1}\left\Vert f\right\Vert _{p}\leq\left\Vert \left(  \sum_{\rho_{\alpha}%
\in\mathcal{I}_{\alpha}}\left\vert S_{\rho_{\alpha}}(f)\right\vert
^{2}\right)  ^{1/2}\right\Vert _{p}\leq C_{2}\left\Vert f\right\Vert _{p}.
\]

\end{theorem}

\begin{proof}
Without loss of generality, suppose that $\alpha\in(0,\pi)$ and $\rho_{\alpha
}=[a_{\alpha},b_{\alpha}]$, where $a_{\alpha}=a\sin\alpha$, $b_{\alpha}%
=b\sin\alpha$ and $a<b$. Then
\[
\chi_{\rho_{\alpha}}(x)=\frac{\mathrm{sgn}(x-a_{\alpha})-\mathrm{sgn}%
(x-b_{\alpha})}{2}.
\]
For $f\in L^{2}(\mathbb{R})\cap L^{p}(\mathbb{R})$, by (\ref{eq:Sp}) we have
\begin{align*}
\mathcal{F}_{\alpha}(S_{\rho_{\alpha}}f)(x)  &  =\chi_{\rho_{\alpha}%
}(x)\left(  \mathcal{F}_{\alpha}f\right)  (x)\\
&  =\frac{i}{2}\left[  \left(  -i\mathrm{sgn}(x-a_{\alpha})\right)  -\left(
-i\mathrm{sgn}(x-b_{\alpha})\right)  \right]  \left(  \mathcal{F}_{\alpha
}f\right)  (x)\\
&  =\frac{i}{2}\left[  \left(  -i\mathrm{sgn}(x-a_{\alpha})\left(
\mathcal{F}_{\alpha}f\right)  (x)\right)  -\left(  -i\mathrm{sgn}(x-b_{\alpha
})\left(  \mathcal{F}_{\alpha}f\right)  (x)\right)  \right] \\
&  =\frac{i}{2}\left[  \tau_{a_{\alpha}}\left(  -i\mathrm{sgn}x\cdot
\tau_{-a_{\alpha}}\left(  \mathcal{F}_{\alpha}f\right)  (x)\right)
-\tau_{b_{\alpha}}\left(  -i\mathrm{sgn}x\cdot\tau_{-b_{\alpha}}\left(
\mathcal{F}_{\alpha}f\right)  (x)\right)  \right]  ,
\end{align*}
where $\tau_{s}f(x)=f(x-s)$. In view of the decomposition (\ref{eq:def}), we
have
\begin{align*}
\tau_{-a_{\alpha}}\left(  \mathcal{F}_{\alpha}f\right)  (x)  &  =A_{\alpha
}\tau_{-a_{\alpha}}\left[  e^{i\pi x^{2}\cot\alpha}\mathcal{F}[e^{i\pi
t^{2}\cot\alpha}f(t)](x\csc\alpha)\right] \\
&  =A_{\alpha}\sin\alpha\cdot\tau_{-a_{\alpha}}\left[  e^{i\pi x^{2}\cot
\alpha}\mathcal{F}[e^{i\pi t^{2}\sin\alpha\cos\alpha}f(t\sin\alpha)](x)\right]
\\
&  =A_{\alpha}\sin\alpha\cdot e^{i\pi(x+a_{\alpha})^{2}\cot\alpha}\cdot
\tau_{-a_{\alpha}}\mathcal{F}[e^{i\pi t^{2}\sin\alpha\cos\alpha}f(t\sin
\alpha)](x)\\
&  =A_{\alpha}\sin\alpha\cdot e^{i\pi(x+a_{\alpha})^{2}\cot\alpha}%
\mathcal{F}[e^{2\pi i(-a_{\alpha})t}e^{i\pi t^{2}\sin\alpha\cos\alpha}%
f(t\sin\alpha)](x).
\end{align*}
Recall that $\mathcal{F(H}f)(x)=-i\mathrm{sgn}x\mathcal{F(}f)(x)$. Hence,%
\begin{align*}
&  -i\mathrm{sgn}x\cdot\tau_{-a_{\alpha}}\left(  \mathcal{F}_{\alpha}f\right)
(x)\\
&  =A_{\alpha}\sin\alpha\cdot e^{i\pi(x+a_{\alpha})^{2}\cot\alpha}\left(
-i\mathrm{sgn}x\mathcal{F}[e^{2\pi i(-a_{\alpha})t}e^{i\pi t^{2}\sin\alpha
\cos\alpha}f(t\sin\alpha)](x)\right) \\
&  =A_{\alpha}\sin\alpha\cdot e^{i\pi(x+a_{\alpha})^{2}\cot\alpha}%
\mathcal{F}\left[  \mathcal{H}\left(  e^{2\pi i(-a_{\alpha})(\cdot)}%
e^{i\pi(\cdot)^{2}\sin\alpha\cos\alpha}f((\cdot)\sin\alpha)\right)
(t)\right]  (x).
\end{align*}
Thus%
\begin{align*}
&  \tau_{a_{\alpha}}\left(  -i\mathrm{sgn}x\cdot\tau_{-a_{\alpha}}\left(
\mathcal{F}_{\alpha}f\right)  (x)\right) \\
&  =A_{\alpha}\sin\alpha\cdot e^{i\pi x^{2}\cot\alpha}\tau_{a_{\alpha}%
}\mathcal{F}\left[  \mathcal{H}\left(  e^{2\pi i(-a_{\alpha})(\cdot)}%
e^{i\pi(\cdot)^{2}\sin\alpha\cos\alpha}f((\cdot)\sin\alpha)\right)
(t)\right]  (x)\\
&  =A_{\alpha}\sin\alpha\cdot e^{i\pi x^{2}\cot\alpha}\mathcal{F}\left[
e^{2\pi ia_{\alpha}t}\mathcal{H}\left(  e^{2\pi i(-a_{\alpha})(\cdot)}%
e^{i\pi(\cdot)^{2}\sin\alpha\cos\alpha}f((\cdot)\sin\alpha)\right)
(t)\right]  (x)\\
&  =A_{\alpha}e^{i\pi x^{2}\cot\alpha}\mathcal{F}\left[  e^{2\pi ia_{\alpha
}(\cdot)\csc\alpha}\mathcal{H}\left(  e^{-2\pi ia_{\alpha}(\cdot)\csc\alpha
}e^{i\pi(\cdot)^{2}\cot\alpha}f\right)  \right]  (x\csc\alpha)\\
&  =\mathcal{F}_{\alpha}\left[  e^{-i\pi(\cdot)^{2}\cot\alpha}e^{2\pi
ia(\cdot)}\mathcal{H}\left(  e^{-2\pi ia(\cdot)}e^{i\pi(\cdot)^{2}\cot\alpha
}f\right)  \right]  (x).
\end{align*}
The definition of the fractional Hilbert transform (\ref{eq:HT}) implies that%
\[
\mathcal{H}_{\alpha}f=e^{-i\pi(\cdot)^{2}\cot\alpha}\mathcal{H}\left(
e^{i\pi(\cdot)^{2}\cot\alpha}f\right)  .
\]
Therefore,%
\[
\tau_{a_{\alpha}}\left(  -i\mathrm{sgn}x\cdot\tau_{-a_{\alpha}}\left(
\mathcal{F}_{\alpha}f\right)  (x)\right)  =\mathcal{F}_{\alpha}\left[  e^{2\pi
ia(\cdot)}\mathcal{H}_{\alpha}\left(  e^{-2\pi ia(\cdot)}f\right)  \right]
(x),
\]
and similarly,%
\[
\tau_{b_{\alpha}}\left(  -i\mathrm{sgn}x\cdot\tau_{-b_{\alpha}}\left(
\mathcal{F}_{\alpha}f\right)  (x)\right)  =\mathcal{F}_{\alpha}\left[  e^{2\pi
ib(\cdot)}\mathcal{H}_{\alpha}\left(  e^{-2\pi ib(\cdot)}f\right)  \right]
(x).
\]
Consequently,%
\begin{align*}
   \mathcal{F}_{\alpha}(S_{\rho_{\alpha}}f)(x)
&  =\frac{i}{2}\left[  \mathcal{F}_{\alpha}\left[  e^{2\pi ia(\cdot
)}\mathcal{H}_{\alpha}\left(  e^{-2\pi ia(\cdot)}f\right)  \right]
(x)-\mathcal{F}_{\alpha}\left[  e^{2\pi ib(\cdot)}\mathcal{H}_{\alpha}\left(
e^{-2\pi ib(\cdot)}f\right)  \right]  (x)\right] \\
&  =\frac{i}{2}\mathcal{F}_{\alpha}\left[  e^{2\pi ia(\cdot)}\mathcal{H}%
_{\alpha}\left(  e^{-2\pi ia(\cdot)}f\right)  -e^{2\pi ib(\cdot)}%
\mathcal{H}_{\alpha}\left(  e^{-2\pi ib(\cdot)}f\right)  \right]  (x).
\end{align*}
Namely,%
\begin{equation}
S_{\rho_{\alpha}}(f)(x)=\frac{i}{2}\left[  e^{2\pi iax}\mathcal{H}_{\alpha
}\left(  e^{-2\pi ia(\cdot)}f\right)  (x)-e^{2\pi ibx}\mathcal{H}_{\alpha
}\left(  e^{-2\pi ib(\cdot)}f\right)  (x)\right]  . \label{eq:L-P1}%
\end{equation}
Since $\mathcal{H}_{\alpha}$ is bounded from $L^{p}(\mathbb{R})$ to
$L^{p}(\mathbb{R})$, $S_{\rho_{\alpha}}$ can be extended to be a bounded
operator on $L^{p}(\mathbb{R})$.

Finally, the classical partial summation operator $S_{\rho}$ corresponding to
$\rho=[a,b]$ is defined by%
\[
\mathcal{F}(S_{\rho}f)(x)=\chi_{\rho}(x)\left(  \mathcal{F}f\right)  (x),
\]
and $S_{\rho}$ can be expressed as (refer to \cite[Example 5.4.6]{Ding})
\begin{equation}
S_{\rho}(f)(x)=\frac{i}{2}\left[  e^{2\pi iax}\mathcal{H}\left(  e^{-2\pi
ia(\cdot)}f\right)  (x)-e^{2\pi ibx}\mathcal{H}\left(  e^{-2\pi ib(\cdot
)}f\right)  (x)\right]  . \label{eq:L-P2}%
\end{equation}
Comparing (\ref{eq:L-P1}) and (\ref{eq:L-P2}) and applying the classical
Littlewood-Paley theorem to (\ref{eq:L-P1}), we   easily conclude that
Theorem \ref{th:L-P} holds.
\end{proof}

\section{Applications to chirps}

\label{sect:ex}

FRFT seems to be a more effective tool than the classical Fourier transform in
frequency spectrum analysis of chirp signals. For example, let%
\[
u(t)=\left\{
\begin{array}
[c]{cc}%
\frac{e^{-\pi it^{2}}}{\sqrt{\left\vert t\right\vert }}, & 0<\left\vert
t\right\vert <1,\\
\frac{e^{-\pi it^{2}}}{t^{2}}, & \left\vert t\right\vert \geq1.
\end{array}
\right.
\]
Then $u$ is a chirp signal and $u\in L^{1}(\mathbb{R})$ but $u\notin
L^{2}(\mathbb{R})$. The real and imaginary part graphs of $u(t)$ in time
domain are shown in Fig. \ref{fig:fi}.

\begin{figure}[htb]
\centering
\subfigure[real part graph of $u(t)$]{
\includegraphics[width=0.35\linewidth]{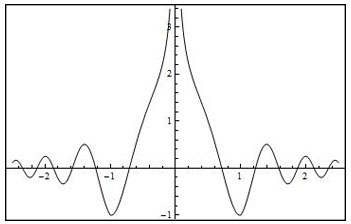}} \hspace{0.5em}
\subfigure[imaginary part graph of $u(t)$]{
\includegraphics[width=0.35\linewidth]{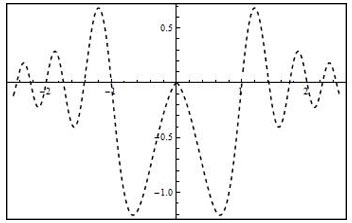}} \caption{real and imaginary
part graphs of $u(t)$ in time domain.}%
\label{fig:fi}%
\end{figure}

Consider the FRFT of $u$ of order $\pi/4$:
\[
(\mathcal{F}_{\pi/4}u)(w)=2e^{i\pi w^{2}}\left(  \frac{\text{C}\left(
2^{5/4}\sqrt{\left\vert w\right\vert }\right)  }{2^{1/4}\sqrt{\left\vert
w\right\vert }}-\sqrt{2}\pi^{2}\left\vert w\right\vert +2\sqrt{2}\pi
w\operatorname{Si}\left(  2\sqrt{2}\pi w\right)  +\cos\left(  2\sqrt{2}\pi
w\right)  \right)  ,
\]
where $\operatorname{C}(x)$ and $\operatorname{Si}(x)$ denote the Fresnel
integral and sine integral, respectively. Namely,%
\begin{eqnarray*}
\text{C}(x)&=&\int_{0}^{x}\sin t^{2}\text{d}t=\sum\limits_{n=0}^{\infty}%
(-1)^{n}\frac{x^{4n+1}}{(2n)!(4n+1)}, \\
\operatorname{Si}(x)&=&\int_{0}^{x}\frac{\sin t}{t}\text{d}t=\sum\limits_{n=0}%
^{\infty}(-1)^{n}\frac{x^{2n+1}}{(2n+1)!(2n+1)}.
\end{eqnarray*}
The real and imaginary part graphs of $(\mathcal{F}_{\pi/4}u)(w)$ in frequency
domain are shown in Fig \ref{fig:gi}.

\begin{figure}[htb]
\centering
\subfigure[real part graph of $(\mathcal{F}_{\pi/4}u)(w)$]{
\includegraphics[width=0.35\linewidth]{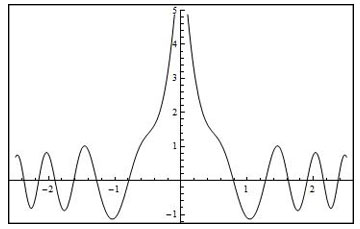}} \hspace{0.5em}
\subfigure[imaginary part graph of $(\mathcal{F}_{\pi/4}u)(w)$]{
\includegraphics[width=0.35\linewidth]{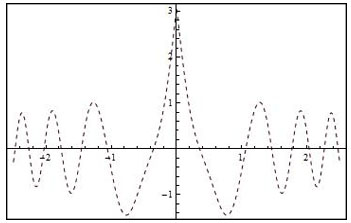}} \caption{real and imaginary
part graphs of $(\mathcal{F}_{\pi/4}u)(w)$ in frequency domain.}%
\label{fig:gi}%
\end{figure}

It is obvious that $\mathcal{F}_{\pi/4}u\notin L^{1}(\mathbb{R})$. The inverse
FRFT%
\begin{equation}
\int_{-\infty}^{+\infty}(\mathcal{F}_{\pi/4}u)(w)K_{-\pi/4}(x,t)\mathrm{d}w
\label{ex}%
\end{equation}
do not make sense. In order to recover the original signal $u(t)$, we should
use the approximating method. Fig. \ref{fig:ue} shows the Abel means of the
integral (\ref{ex}) with $\varepsilon=1,0.1,0.01$, that is,%
\[
u_{\varepsilon}(t)=\int_{-\infty}^{+\infty}(\mathcal{F}_{\pi/4}u)(w)K_{-\pi
/4}(x,t)e^{-2\pi\varepsilon\left\vert \csc\alpha\right\vert \left\vert
w\right\vert }\mathrm{d}w.
\]
By Theorems \ref{cor:L1}-\ref{cor:point} and Corollary \ref{cor:p-w} we know that $u_{\varepsilon
}(t)\rightarrow u(t)$ for a.e. $t\in\mathbb{R}$ as $\varepsilon\rightarrow0$.

\begin{figure}[htb]
\centering
\subfigure[real part graph of $u_\varepsilon(t)$]{
\includegraphics[width=0.35\linewidth]{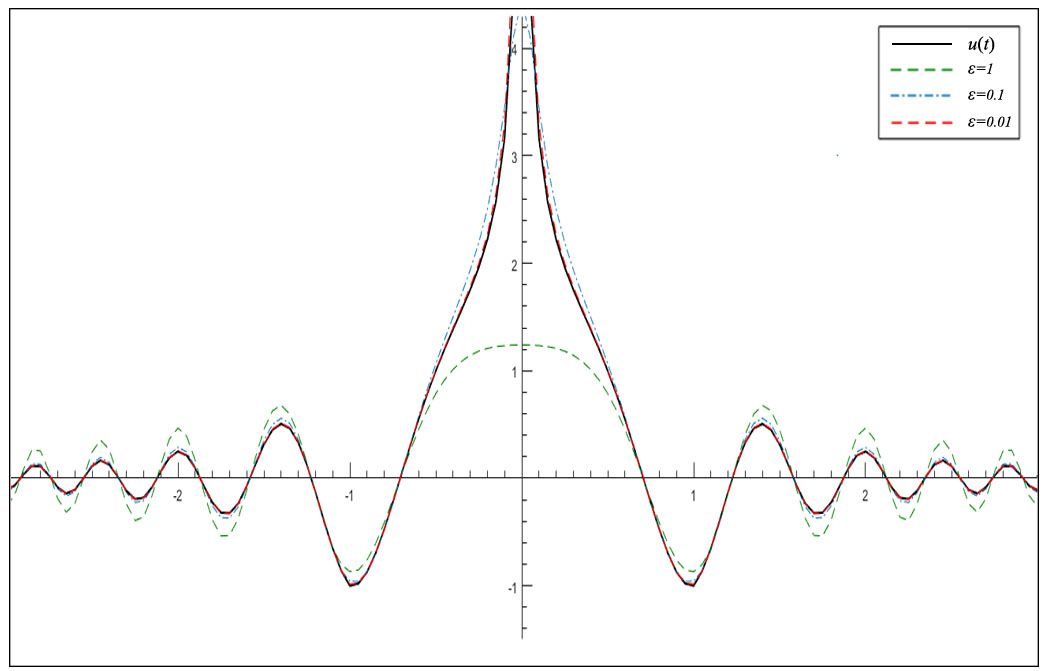}} \hspace{0.5em}
\subfigure[imaginary part graph of $u_\varepsilon(t)$]{
\includegraphics[width=0.35\linewidth]{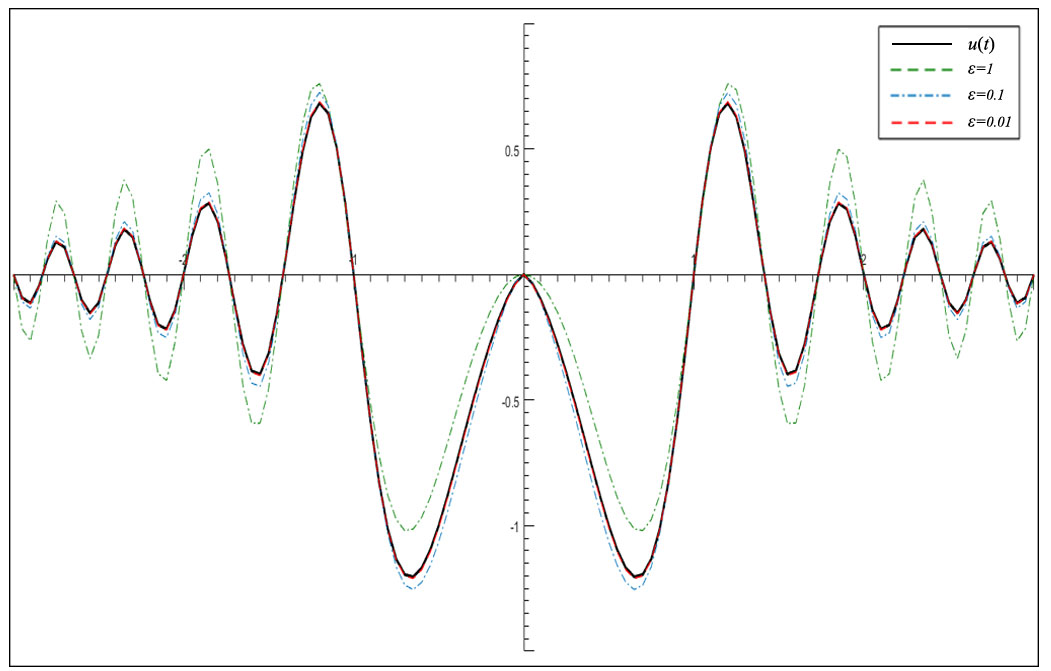}} \caption{real and imaginary
part graphs of $u_{\varepsilon}(t)$}%
\label{fig:ue}%
\end{figure}


\end{document}